\documentclass[12pt, reqno]{amsart}
\usepackage{amsmath, amsfonts, amssymb}
\newcommand{\dal}{\square}
\newcommand{\R}{{\mathbb R}}
\newcommand{\N}{{\mathbb N}}
\newcommand{\C}{{\mathbb C}}
\newcommand{\Sp}{{\mathbb S}}
\newcommand{\ve}{\varepsilon}
\newcommand{\pa}{\partial}
\newcommand{\jb}[1]{\left\langle #1 \right\rangle}
\DeclareMathOperator{\supp}{\rm supp}
\DeclareMathOperator{\real}{\rm Re}
\theoremstyle{plain}
\newtheorem{theorem}{Theorem}[section]

\newtheorem{corollary}[theorem]{Corollary}
\newtheorem{lemma}[theorem]{Lemma}
\theoremstyle{remark}

\theoremstyle{definition}

\numberwithin{equation}{section}
\title[Asymptotic behavior for systems of nonlinear wave equations]
{Asymptotic behavior for systems of nonlinear wave equations
with multiple propagation speeds in three space dimensions}

\author[S.~Katayama]{Soichiro Katayama}
\address{Department of Mathematics, Wakayama University, 930 Sakaedani, Wakayama 640-8510, Japan\\
Tel: +81-73-457-7343}
\email{katayama@center.wakayama-u.ac.jp}


\date{}

\begin{document}

\begin{abstract}
We consider the Cauchy problem for systems of nonlinear wave equations 
with multiple propagation speeds in three space dimensions.
Under the null condition for such systems, 
the global existence of small amplitude solutions is known. 
In this paper, 
we will show that the global solution 
is asymptotically free in the energy sense, 
by obtaining the asymptotic pointwise behavior of the derivatives
of the solution. Nonetheless we can also show that the pointwise behavior of the solution itself may be quite different from that of the free solution.
In connection with the above results, a theorem is also developed to characterize asymptotically free solutions for wave equations in arbitrary space dimensions.
\end{abstract}

\maketitle
\section{Introduction}
We consider the Cauchy problem for a system of nonlinear wave equations
of the following type with small initial data:
\begin{align}
 \label{OurSystem}
& \dal_{c_j} u_j(t,x)=F_j\bigl(\pa u (t,x), \pa^2 u(t,x)\bigr),
\quad (t,x)\in(0,\infty)\times \R^3,\\
 \label{OurData}
 & u_j(0,x)=\ve f_j(x), \ (\pa_t u_j)(0,x)=\ve g_j(x), \quad x\in \R^3
\end{align}
for $j=1, \ldots, N$,
where $\dal_c=\pa_t^2-c^2\Delta_x=\pa_t^2-c^2\sum_{k=1}^3\pa_{x_k}^2$ for $c>0$, and  
the propagation speeds $c_1,\ldots, c_N$ are positive constants,
while
$\pa u$ and $\pa^2 u$ denote the first and second derivatives of 
$u=(u_l)_{1\le l\le N}$, respectively. More specifically, we write 
$$
\pa u=(\pa_a u_l)_{1\le l\le N,\, 0\le a\le 3},\ \pa^2 u=(\pa_a\pa_b u_l)_{
1\le l\le N,\, 0\le a,b\le 3}
$$
with the notation
$$
 \pa_0:=\pa_t=\frac{\pa}{\pa t}, \text{ and }\pa_k:=\pa_{x_k}=\frac{\pa}{\pa x_k} 
\text{ for $k=1,2,3$.}
$$
We assume that $f=(f_j)_{1\le j\le N}, g=(g_j)_{1\le j\le N}\in C^\infty_0(\R^3;\R^N)$.
$\ve$ in \eqref{OurData} is a small and positive parameter.
$F_j$ can include nonlinear terms of higher order, but for simplicity we
assume that $F_j$ can be written as
\begin{equation}
\label{Non01}
F_j=\sum_{k, l=1}^N
\left(\sum_{a,b,b'=0}^3
p_{jkl}^{abb'} (\pa_a u_k)(\pa_b\pa_{b'} u_l)+\sum_{a,b=0}^3 q_{jkl}^{ab}(\pa_a u_k)(\pa_b u_l)\right)
\end{equation}
for $1\le j\le N$
with appropriate real constants $p_{jkl}^{abb'}$ and $q_{jkl}^{ab}$,
where $p_{jkl}^{a00}=0$ for $1\le j,k, l\le N$ and $0\le a\le 3$.
To assure the hyperbolicity, we assume the symmetry condition
\begin{equation}
\label{symmetry}
p_{jkl}^{abb'}=p_{lkj}^{abb'}=p_{jkl}^{ab'b}
\end{equation}
for $1\le j, k, l \le N$, and $0\le a,b,b' \le 3$.

For general quadratic nonlinearity, the classical solution to \eqref{OurSystem}--\eqref{OurData}
may blow up in finite time for some $(f,g)$ no matter how small $\ve $ is. 
Klainerman \cite{Kla86} introduced a sufficient condition for small data global existence
for the single speed case where $c_1=\cdots=c_N(=1)$ (see also Christodoulou \cite{Chr86}).
This sufficient condition, called the {\it null condition}, was extended to the multiple speed case: To simplify the description, we assume that the speeds are distinct, and that
\begin{equation}
\label{distinct speeds}
0<c_1<c_2<\cdots<c_N.
\end{equation}
Let the constants $p_{jkl}^{abb'}$ and $q_{jkl}^{ab}$ be from  \eqref{Non01}.
We say that
the null condition (associated with the speeds $c_1,\ldots, c_N$) is satisfied
if we have
\begin{equation}
\label{NullC}
\sum_{a,b,b'=0}^3 p_{jjj}^{abb'} X_aX_bX_{b'}=\sum_{a,b=0}^3 q_{jjj}^{ab}X_aX_b
=0,\quad X\in{\mathcal N}_j,\ 1\le j\le N,
\end{equation}
where 
${\mathcal N}_j:=\{X=(X_a)_{0\le a\le 3}\in \R^4;\, X_0^2-c_j^2\sum_{k=1}^3X_k^2=0\}$.
Small data global existence under the null condition for the multiple speed case
was obtained by Yokoyama~\cite{Yok00}
(see also Sideris-Tu \cite{SidTu01}, Sogge~\cite{Sog03}, and Hidano~\cite{Hid04};
see Kubota-Yokoyama \cite{Kub-Yok01}, the author \cite{Kat04:02},
and Metcalfe-Nakamura-Sogge~\cite{MetNaSo05b} for the case where nonlinearity of higher order depends not only on $(\pa u, \pa^2 u)$, but also on $u$).
We introduce the {\it null forms}
\begin{align}
\label{NullForm01}
Q_0(\varphi,\psi;c)=&(\pa_t\varphi)(\pa_t\psi)-c^2\sum_{k=1}^3 (\pa_k\varphi)(\pa_k\psi),\\
\label{NullForm02}
Q_{ab}(\varphi,\psi)=&(\pa_a\varphi)(\pa_b \psi)-(\pa_b\varphi)(\pa_a\psi),\quad
0\le a, b\le 3.
\end{align}
Then it is shown in \cite{Yok00} that the quadratic nonlinearity satisfying the null condition can be written as
\begin{equation}
\label{Non02}
F_j(\pa u,\pa^2 u)=N_j(\pa u, \pa^2 u)+R_j^{\rm I}(\pa u, \pa^2 u)
{}+R_j^{\rm II}(\pa u, \pa^2 u),
\end{equation}
where
\begin{align}
N_j=& \sum_{0\le |\alpha|\le 1} A_j^{\alpha} Q_0(u_j,\pa^\alpha u_j; c_j)
{}+\sum_{a,b,b'=0}^3 B_j^{abb'}Q_{ab}(u_j, \pa_{b'}u_j), \label{DefNullT}\\
R_j^{\rm I}=& \sum_{\{(k,l); k\ne l\}}\left(\sum_{a,b,b'=0}^3
p_{jkl}^{abb'} (\pa_a u_k)(\pa_b\pa_{b'} u_l)+\sum_{a,b=0}^3 q_{jkl}^{ab}(\pa_a u_k)(\pa_b u_l)\right), \label{DefNonResI}\\
R_j^{\rm II}=& \sum_{\{k; k\ne j\}}\left(\sum_{a,b,b'=0}^3
p_{jkk}^{abb'} (\pa_a u_k)(\pa_b\pa_{b'} u_k)+\sum_{a,b=0}^3 q_{jkk}^{ab}(\pa_a u_k)(\pa_b u_k)\right)
\label{DefNonResII}
\end{align}
with some constants $A_j^{\alpha}$ and $B_j^{abb'}$ ($p_{jkl}^{abb'}$ and $q_{jkl}^{ab}$ are from \eqref{Non01}).
Here $\pa=(\pa_a)_{0\le a\le 3}$, and $\alpha=(\alpha_a)_{0\le a\le 3}$ is a multi-index.
We refer to terms involved in $R^{\rm I}=(R_j^{\rm I})_{1\le j\le N}$ and 
$R^{\rm II}=(R_j^{\rm II})_{1\le j\le N}$
as the {\it non-resonant terms of type I} and {\it type II}, respectively.

Now we turn our attention to the asymptotic behavior of the global solutions.
Let $c>0$.
It is known that if $G\in L^1\bigl((0,\infty);L^2(\R^3)\bigr)$,
then the 
solution $v$ to $\dal_c v(t,x)=G(t,x)$ is
{\it asymptotically free in the energy sense}, that is to say,
there exists a 
solution $v^+$ to the free wave equation
$\dal_c v^+=0$ such that
$$
\lim_{t\to \infty} \|(v-v^+)(t)\|_{E,c}=0,
$$
where the energy norm $\|\varphi(t)\|_{E,c}$ is defined by
$$
\|\varphi(t)\|_{E,c}^2=\frac{1}{2}\int_{\R^3} \left(\frac{1}{c^2}|\pa_t\varphi(t,x)|^2+|\nabla_x \varphi(t,x)|^2\right)dx
$$
with $\nabla_x=(\pa_1,\pa_2,\pa_3)$.
For the single speed case, 
investigating the proof in \cite{Kla86}, we have
$$ 
 F_j(\pa u, \pa^2 u)\in L^1\bigl((0, \infty); L^2(\R^3)\bigr) 
$$
under the null condition because of the extra decay for the null forms, 
and as an immediate consequence we see that the global solution $u$ is 
asymptotically free in the energy sense. 
As for the multiple speed case, by 
the estimates obtained in \cite{Kub-Yok01} (cf.~\eqref{Basic01a} and \eqref{Basic01b} below), 
it is easy to see that
\begin{equation}
\label{AsympCond}
N_j(\pa u,\pa^2 u)+R_j^{\rm I}(\pa u,\pa^2 u)\in L^1\bigl((0,\infty);L^2(\R^3)\bigr)
\end{equation}
for the global solution $u$ to \eqref{OurSystem}--\eqref{OurData},
because we can expect some gain in the decay rate for the null forms and the non-resonant
terms of type I, compared to general quadratic nonlinearity (see \eqref{Est02} and \eqref{Est03} below).
Therefore, if $R^{\rm II}\equiv 0$ in \eqref{Non02}, then 
\eqref{AsympCond} implies that the 
global solution $u$
 is asymptotically free
in the energy sense; namely for $j=1,\ldots, N$,
there exists a 
solution $u_j^+$ of the free wave equation $\dal_{c_j} u_j^+=0$ such that
$\lim_{t\to\infty}\|(u_j-u_j^+)(t)\|_{E,c_j}=0$.
By contrast, there is no explicit gain in the decay rate for the non-resonant terms of type II unless they can be written in terms of the null forms, and
we cannot expect 
$R_j^{\rm II}(\pa u,\pa^2 u)\in L^1\bigl((0,\infty);L^2(\R^3)\bigr)$ in general (see \eqref{Est04} below),
although its influence is weak enough for the solution
to exist globally.
Hence it is not clear
whether the global solution $u$ is asymptotically free or not
when $R^{\rm II}$ is not written in terms of the null forms.
Our aim in this paper is to determine the asymptotic behavior
in the presence of the non-resonant terms of type II by modifying the method
developed in \cite{Kat11}.
\section{Main Results}
\label{MainResults}
\subsection{Asymptotically free functions in the energy sense.}
To begin with, we will characterize the asymptotically free
functions in the energy sense. 
We only need the three space dimensional result
for an application in this paper,
but we consider the general space dimensional case here for future applications.

Let $n$ be a positive integer.
Let $\dot{H}^1(\R^n)$ be the completion of $C^\infty_0(\R^n)$
with respect to the norm $\|\varphi\|_{\dot{H}^1(\R^n)}=\|\nabla_x\varphi\|_{L^2(\R^n)}$.
We put
\begin{equation}
X_n:=C\bigl([0,\infty); \dot{H}^1(\R^n)\bigr)
     \cap C^1\bigl([0,\infty); L^2(\R^n)\bigr).
\label{FunctionSpace}
\end{equation}
We say that a function $v=v(t,x)\in X_n$ is {\it asymptotically free
in the energy sense associated with the speed $c$},
if there is $(v_0^+, v_1^+)\in \dot{H}^1(\R^n)\times L^2(\R^n)$
such that
$$
\lim_{t\to\infty} \|v(t,\cdot)-v^+(t,\cdot)\|_{E,c}=0,
$$
where $v^+\in X_n$ is a unique solution to
$$
\dal_c v^+(t,x)=0,\quad (t,x)\in(0,\infty)\times \R^n
$$ 
with initial data
$\bigl(v^+(0),\pa_t v^+(0)\bigr)=(v_0^+, v_1^+)$,
and the energy norm (associated with the speed $c$) is given by
$$
\|\varphi(t,\cdot)\|_{E,c}^2:=\frac{1}{2}\int_{\R^n}
\left(\frac{1}{c^2}|\pa_t \varphi(t,x)|^2+|\nabla_x \varphi(t,x)|^2\right)dx.
$$
Here $\dal_c=\pa_t^2-c^2\sum_{k=1}^n \pa_k^2$ with the notation
$\pa_k=\pa/\pa x_k$ for $1\le k\le n$.
Note that we do not suppose that $v$ is a solution to some wave equation here.

For $c>0$ and $x\in \R^n\setminus\{0\}$, we define
\begin{equation}\label{speedconst}
\vec{\omega}_c(x)=\bigl({\omega}_{c,a}(x)\bigr)_{0\le a\le n}:=(-c, |x|^{-1}x).
\end{equation}
\begin{theorem} \label{AFSE}
Let $n\ge 2$ and $c>0$. A function $v\in X_n$
is asymptotically free in the energy sense associated with
the speed $c$, if and only if there is a function
$V=V(\sigma, \omega)\in L^2(\R\times \Sp^{n-1})$ such that
$$
\lim_{t\to\infty} \bigl\|\pa v(t,\cdot)-\vec{\omega}_c(\cdot){V}^\sharp(t,\cdot)
\bigr\|_{L^2(\R^n)}=0,
$$
where $\pa=(\pa_0, \pa_1, \ldots, \pa_n)$, and
$$
{V}^\sharp(t,x):=|x|^{-(n-1)/2}V(|x|-ct, |x|^{-1}x),\quad (t,x)\in (0,\infty)\times(\R^n\setminus\{0\}).
$$
\end{theorem}

This theorem for $n=3$ was implicitly proved in \cite{Kat11}.
The author believes that this result for general space dimensions is a new observation,
but it possibly has already appeared in some literature.
We will give the proof in Section~\ref{ProofAFSE}.
See \cite{KatMurSun11} for an application of Theorem~\ref{AFSE}
to semilinear wave equations in two space dimensions.
Another application in three space dimensions can be found in \cite{KatMatSun12}.
\subsection{Asymptotic behavior of the global solutions}
Next we examine the asymptotic behavior of the global solutions to
\eqref{OurSystem}--\eqref{OurData}.
For $h\in C^\infty(\R^3)$, we define its Radon transform ${\mathcal R}[h]$ 
by
$$
{\mathcal R}[h](\sigma, \omega)=\int_{y\cdot\omega=\sigma} h(y) dS(y),
$$
where $dS(y)$ denotes the surface element on the plane $\{y\cdot\omega=\sigma\}$.
Now, restricting our attention to the three space dimensional case,
we introduce the Friedlander radiation field (its definition
in general space dimensions will be given in \eqref{DefGeneralFried} below):
For $(\varphi, \psi)\in C^\infty_0(\R^3)\times C^\infty_0(\R^3)$, we define
the {\it Friedlander radiation field}
\begin{equation}
\label{Friedlander3D}
{\mathcal F}_0[\varphi,\psi](\sigma, \omega)=\frac{1}{4\pi}
\bigl({\mathcal R}[\psi](\sigma, \omega)-(\pa_\sigma {\mathcal R}[\varphi])(\sigma,\omega)\bigr)
\end{equation}
for $(\sigma, \omega)\in \R\times \Sp^2$. 
For $z\in \R^d$ with a positive integer $d$, we put $\jb{z}:=\sqrt{1+|z|^2}$. 
\begin{theorem}\label{Main01}
Suppose that \eqref{Non01}, \eqref{symmetry}, \eqref{distinct speeds},
and the null condition 
\eqref{NullC} are fulfilled. 
Let $\pa=(\pa_0, \pa_1, \pa_2, \pa_3)$, and $\vec{\omega}_c(x)$ be given by
\eqref{speedconst}.
\\
{\rm (1)}
We fix small $\delta>0$.
Then for any $f,g\in C^\infty_0(\R^3; \R^N)$ and sufficiently small $\ve>0$,
there is a function $P=(P_j)_{1\le j\le N}$ of
$(\sigma,\omega)\in \R\times \Sp^2$ 
such that
\begin{align}
|x| \pa u_j(t, x)=& \ve \vec{\omega}_{c_j}(x) P_j(|x|-c_jt, |x|^{-1}x)
\nonumber\\
&{}+O\bigl(\ve \jb{t+|x|}^{-1+\delta}\jb{c_jt-|x|}^{-\delta}\bigr)
\label{Main11}
\end{align}
for $(t,x)\in [0,\infty)\times(\R^3\setminus\{0\})$ and $1\le j\le N$,
where $u=(u_j)_{1\le j\le N}$ is 
the global solution to the Cauchy problem \eqref{OurSystem}--\eqref{OurData}.
Moreover we have
\begin{equation}
\label{Main12}
P_j(\sigma, \omega)=\pa_\sigma{\mathcal F}_0[f_j,c_j^{-1}g_j](\sigma, \omega)
+O\bigl(\ve \jb{\sigma}^{-1}\bigr),\quad
(\sigma, \omega)\in \R\times \Sp^2.
\end{equation}
{\rm (2)} 
We further assume that $R^{\rm II}=(R_j^{\rm II})_{1\le j\le N}$ has the null structure,
that is to say
$$
R_j^{\rm II}=\sum_{\{k;k\ne j\}}
\left(\sum_{0\le |\alpha|\le 1} A_{jk}^{\alpha} Q_0(u_k, \pa^\alpha u_k; c_k)
+\sum_{a, b, b'=0}^3 B_{jk}^{abb'}Q_{ab}(u_k, \pa_{b'}u_k)\right)
$$
for $1\le j\le N$ with some constants $A_{jk}^\alpha$ and $B_{jk}^{abb'}$.
Then for any $f,g\in C^\infty_0(\R^3,\R^N)$ and sufficiently small $\ve>0$,
there is a function $U=(U_j)_{1\le j\le N}$ of
$(\sigma, \omega)\in \R\times \Sp^2$ such that
\begin{align}
\label{Main31}
|x| u_j(t, x)= & \ve U_j(|x|-c_jt, |x|^{-1}x)+O\bigl(\ve \jb{t+|x|}^{-1}\log(2+t)\bigr), \\
\label{Main32}
|x| \pa u_j(t, x)= & \ve \vec{\omega}_{c_j}(x) (\pa_\sigma U_j)(|x|-c_jt, |x|^{-1}x)
\nonumber\\
& {}+O\bigl(\ve \jb{t+|x|}^{-1}\jb{c_jt-|x|}^{-1}\bigr)
\end{align}
for $(t,x)\in [0,\infty)\times(\R^3\setminus\{0\})$ and $1\le j\le N$,
where $u=(u_j)_{1\le j\le N}$ is 
the global solution to the Cauchy problem \eqref{OurSystem}--\eqref{OurData}.
Moreover we have
\begin{equation}
\label{Main33}
\pa_\sigma^m U_j(\sigma, \omega)=\pa_\sigma^m{\mathcal F}_0[f_j,c_j^{-1}g_j](\sigma, \omega)+O\bigl(\ve \jb{\sigma}^{-1-m}\bigr),\ (\sigma, \omega)\in \R\times \Sp^2
\end{equation}
for $m=0,1$ and $1\le j\le N$.
\end{theorem}

By \eqref{Main11} we see that the asymptotic pointwise behavior
of $\pa u_j$ 
is similar to
that of the derivatives of the free solution 
(see Lemma~\ref{PointwiseFriedlander} below). 
Combining \eqref{Main11} and \eqref{Main12} with Theorem~\ref{AFSE}, 
we see that the solution $u$ is 
asymptotically free in the energy sense.
\begin{corollary}\label{AsympFreeEnergy}
Suppose that \eqref{Non01}, \eqref{symmetry}, \eqref{distinct speeds},
and the null condition 
\eqref{NullC} are fulfilled. 
Then, for any $f,g\in C^\infty_0(\R^3; \R^N)$ and sufficiently small $\ve>0$,
there exist
$f^+=(f_j^+)_{1\le j\le N}\in \dot{H}^1(\R^3;\R^N)$ and $g^+=(g_j^+)_{1\le j\le N}
\in L^2(\R^3;\R^N)$ 
such that
$$
 \lim_{t\to\infty}\|(u_j-u_j^+)(t)\|_{E,c_j}=0,\quad 1\le j\le N,
$$
where $u=(u_j)_{1\le j\le N}$ is 
the global solution to the Cauchy problem \eqref{OurSystem}--\eqref{OurData}, and $u_j^+$ is the solution to $\dal_{c_j}u_j^+=0$ with initial data
$(u_j^+, \pa_t u_j^+)=(f_j^+, g_j^+)$ at $t=0$.
\end{corollary}

Theorem~\ref{Main01} and Corollary~\ref{AsympFreeEnergy} will be proved 
in Section~\ref{P01},
after giving some preliminaries in Section~\ref{P00}.

When $R^{\rm II}$ has the null structure, we see from \eqref{Main31} 
that the solution $u_j$ itself also behaves 
similarly to the free solution.
The next result shows that 
the lack of the estimate corresponding to \eqref{Main31}
in the presence of $R^{\rm II}$ without the null structure is inevitable.
\begin{theorem}\label{NotFree}
Let $0<c_1<c_2$, and let $A_1, A_2$ be real constants.
We consider the Cauchy problem for
\begin{equation}
\label{OurExample}
\begin{cases}
\dal_{c_1}u_1=A_1(\pa_t u_2)^2,\\
\dal_{c_2}u_2=A_2(\pa_t u_1)^2
\end{cases}
\text{ in $(0,\infty)\times \R^3$}.
\end{equation}
If $A_1\ne 0$, then there exist $(f,g)\in C^\infty_0(\R^3;\R^2)\times C^\infty_0(\R^3;\R^2)$, 
$M>0$, $T_0\ge 2$, and $C>0$ such that 
\begin{equation}
\label{Main41}
C^{-1}\ve\bigl(1+\ve\log(2+t)\bigr) \le |r u_1(t,x)|\le C\ve 
\bigl(1+\ve\log(2+t)\bigr)
\end{equation}
for any $(t,x)$ satisfying $T_0\le c_1t\le r(=|x|)\le c_1t+M$,
where $u=(u_1, u_2)$ is the global solution to \eqref{OurExample} with initial data $u=\ve f$ and $\pa_t u=\ve g$ at $t=0$ for sufficiently small $\ve(>0)$.
\end{theorem}

This theorem will be proved in Section~\ref{P02}.
From this result, we see that there is some loss in the decay rate of $u_1$: 
\eqref{Main41} says that $u_1(t,x)$ decays like $(1+t)^{-1}\log(2+t)$ 
along the ray $r=c_1t+\sigma$ for $0\le \sigma \le M$, while
\eqref{Main31} implies that $u_j$ decays like $(1+t)^{-1}$ 
and has the same decay property as the free solution along the ray $r=c_jt+\sigma$
if $R^{\rm II}$ has the null structure. Accordingly we find that 
estimates like \eqref{Main31} cannot hold in the presence of the non-resonant terms of type II without the null structure.

To sum up the results, an interesting character
of the non-resonant terms of type II is revealed:
Their effect is weak enough
for the solution to exist globally (Theorem~1.1 in \cite{Yok00}) and 
to be asymptotically free in the energy sense (Corollary~\ref{AsympFreeEnergy});
however it is strong enough to affect the decay rate of the solution $u$ itself
(Theorem~\ref{NotFree}),
though that of $\pa u$ is not affected (Theorem~\ref{Main01}).

Throughout this paper, various positive constants are denoted by the same letter $C$.
Thus the actual value of $C$ may change line by line.
\section{Asymptotics for Homogeneous Wave Equations}
\label{ProofAFSE}

Our aim in this section is to prove Theorem~\ref{AFSE}.
Most of the necessary materials for this purpose are rather standard, 
but we give the details and proofs to make this section self-contained.
In what follows, we use the formal expression
of writing distributions as if they are functions.
\subsection{Solutions to the free wave equation}
We put $H_0(\R^n)=\dot{H}^1(\R^n)\times L^2(\R^n)$, and we define
\begin{equation}
\label{DefNormH0}
\|(\varphi,\psi)\|_{H_0(\R^n)}^2:=\frac{1}{2}\left(
\|\nabla_x \varphi\|_{L^2(\R^n)}^2+\|\psi\|_{L^2(\R^n)}^2\right).
\end{equation}
$H_0(\R^n)$ can also be understood as the completion of $C^\infty_0(\R^n)\times C^\infty_0(\R^n)$ with respect to the norm $\|\cdot\|_{H_0(\R^n)}$.
Let $c>0$, and consider the Cauchy problem for the free wave equation
\begin{align}
& \dal_c w(t,x)=0, \quad (t,x)\in (0,\infty)\times \R^n,
\label{FreeWaveEq}\\
& \bigl(w(0,x), (\pa_t w)(0,x)\bigr)=\bigl(w_0(x), w_1(x)\bigr),\quad x\in \R^n.
\label{FreeData}
\end{align}
Let $X_n$ be defined by \eqref{FunctionSpace}.
It is known that for $(w_0, w_1)\in H_0(\R^n)$, \eqref{FreeWaveEq}--\eqref{FreeData} admits a unique solution $w\in X_n$,
and we have the conservation of the energy
$$
\|w(t,\cdot)\|_{E,c}=\|w(0,\cdot)\|_{E,c}=\|(w_0, c^{-1}w_1)\|_{H_0(\R^n)}.
$$

For $\real a>-1$, let $\chi_+^a$ be defined by
$$
\chi_+^a(s):=\begin{cases}
\displaystyle \frac{s^a}{\Gamma(a+1)}, & s> 0,\\
0, & s\le 0,
\end{cases}
$$
where $\Gamma$ denotes the Gamma function. 
Then we have $\chi_+^{a}(s)=(\chi_+^{a+1})'(s)$ for $\real a>-1$.
We can extend the definition of the distribution $\chi_+^a$ to all $a\in \C$ so that we have
$\chi_+^{a}(s)=(\chi_+^{a+1})'(s)$ for any $a\in \C$.
Note that $\chi_+^a$ can have its singularity only at $s=0$.
Especially we have
$$
\chi_+^{-k}(s)=\delta^{(k-1)}(s),\quad k\in \N,
$$
where $\delta$ is the Dirac function and $\delta^{(j)}$ denotes its $j$-th derivative.
For a positive integer $m$, 
we define
$$
E_m(t,x)=\frac{1}{2\pi^{(m-1)/2}}\chi_+^{(1-m)/2}(t^2-|x|^2).
$$
Then the solution $w$ to \eqref{FreeWaveEq}--\eqref{FreeData} 
with $(w_0, w_1)\in \bigl(C^\infty_0(\R^n)\bigr)^2$ 
can be written as
\begin{equation}
\label{FreeSolExp}
w(t,x)=c^{-1}\pa_t \bigl(E_n(ct,\cdot)*w_0\bigr)(x)+c^{-1}\bigl(E_n(ct,\cdot)*w_1)(x),
\end{equation}
where the convolution $*$ is taken with respect to $x$-variable
(see H\"ormander \cite[Section~6.2]{Hoe90-01} for instance).

From now on, we suppose that $(w_0, w_1)\in \bigl(C^\infty_0(\R^n)\bigr)^2$,
and that $w_0(x)=w_1(x)=0$ for $|x|\ge M$ with some positive constant $M$.
Since $\supp E_n(t,\cdot)\subset \{x; |x|\le t\}$, it follows from \eqref{FreeSolExp}
that
\begin{equation}
\label{Huygens01}
 w(t,x)=0,\quad |x|\ge ct+M.
\end{equation}
If $n(\ge 3)$ is odd, then $\supp E_n(t,\cdot)=\{x; |x|=t\}$ and we also get
\begin{equation}
\label{Huygens02}
 w(t,x)=0,\quad |x|\le ct-M,
\end{equation}
which is called the {\it (strong) Huygens principle}.
If $n$ is even, \eqref{Huygens02} is not valid in general, but we have
a faster decay away from the light cone $ct=|x|$:
Let $t/2\ge |x|$ and $t\ge 4M$, say. Then $t^2-|x-y|^2\ge 7t^2/16\ge C \jb{t+|x|}^2>0$ for $|y|\le M$. Hence, observing that $\chi^{(1-n)/2}(s)=A_ns^{(1-n)/2}$ for $s>0$ with an
appropriate constant $A_n$, we get
$$
|\pa_{t,x}^\alpha E_n(t,x-y)|\le C_\alpha \jb{t+|x|}^{-|\alpha|+(1-n)},\quad t/2\ge \max\{|x|, 2M\},\ |y|\le M
$$
with a positive constant $C_\alpha$.
Therefore \eqref{FreeSolExp} leads to 
\begin{equation}
\label{Huygens03}
|\pa^\alpha w(t,x)|\le C_\alpha\jb{t+|x|}^{-|\alpha|+(1-n)}, \quad ct/2\ge |x|,
\end{equation}
because \eqref{Huygens03} for $2M\ge ct/2\ge |x|$ is easily shown.

Following the arguments in H\"ormander~\cite[Section~6.2]{Hoe97},
we will obtain a useful expression of 
$E_m(t)*\varphi$ for $\varphi\in C^\infty_0(\R^n)$
with $\varphi(x)=0$ for $|x|\ge M$. 
Note that we have
\begin{equation}
\bigl(E_m(t)*\varphi\bigr)(x)=0, \quad |x|\ge t+M, \label{Huygens01'}
\end{equation}
and
\begin{equation}
\bigl(E_m(t)*\varphi\bigr)(x)=0, \quad |x|\le t-M\text{ when $m(\ge 3)$ is odd} 
\label{Huygens02'}
\end{equation}
as in \eqref{Huygens01} and \eqref{Huygens02}.
Let $x=r\omega$ with $r=|x|$ and $\omega\in \Sp^{n-1}$.
We assume
$$
2M\le \frac{t}{2}\le r\le t+M.
$$
We put $\sigma=r-t$. Then we get $-r\le \sigma\le M$.
Since we have
$$
t^2-|x-y|^2=2r(\omega\cdot y-\sigma)+\sigma^2-|y|^2,
$$
the homogeneity of $\chi_+^{(1-m)/2}$ implies that
\begin{align*}
(2\pi r)^{(m-1)/2}\bigl(E_m(t)*\varphi\bigr)(x)
=& \frac{1}{2}\int_{\R^n} \chi_+^{(1-m)/2}\left(\omega\cdot y-\sigma+
\frac{\sigma^2-|y|^2}{2r}\right)\varphi(y)dy
\\
=& \frac{1}{2}\int_{\R} \chi_+^{(1-m)/2}\left(s-\sigma+
\frac{\sigma^2}{2r}\right){\mathcal G}[\varphi](s, \omega, r^{-1})ds,
\end{align*}
where ${\mathcal G}[\varphi](s, \omega, z)$ is given by
$$
{\mathcal G}[\varphi](s, \omega, z)=\int_{\R^n}\delta\left(s-\omega\cdot y+\frac{|y|^2}{2}z\right)\varphi(y) dy
$$
for $(s, \omega, z)\in \R\times \Sp^{n-1}\times [0, (2M)^{-1}]$.
If we put $\rho=s-\omega\cdot y+|y|^2z/2$, then
$\nabla_y\rho=-\omega+zy$. Since $|\nabla_y\rho|\ge 1-z|y|\ge 1/2$
for $z\in [0, (2M)^{-1}]$ and $|y|\le M$, ${\mathcal G}[\varphi]$
can be written as an integral of a compactly supported function
$\varphi$ over a hyper-surface $\{y\in \R^n; s-\omega\cdot y+|y|^2z/2=0\}$
which smoothly depends on $(s,\omega, z)\in \R\times\Sp^{n-1}\times [0, (2M)^{-1}]$,
and we see that
$$
{\mathcal G}[\varphi]\in C^\infty\left(\R\times \Sp^{n-1}\times \left[0, (2M)^{-1}\right]\right).
$$
We also see that 
if ${\mathcal G}[\varphi](s,\omega,z)\ne 0$ for some $(\omega,z)\in \Sp^{n-1}\times \left[0, (2M)^{-1}\right]$, then we have
$-5M/4\le s\le M$. Indeed, since the assumption implies
$s=\omega\cdot y-(|y|^2/2)z$ for some $y$ and $(\omega, z)$ with $|y|\le M$
and $(\omega, z)\in \Sp^{n-1}\times \left[0, (2M)^{-1}\right]$, we get
$$
-\frac{5M}{4}\le -|y|-\frac{|y|^2}{2}\frac{1}{2M}\le \omega\cdot y-\frac{|y|^2}{2}z(=s)
\le |y|\le M.
$$
For a multi-index $\alpha$ with $|\alpha|=k\le 1$, we have
\begin{equation}
{\mathcal G}[\pa_x^\alpha \varphi](s, \omega, z)=
\pa_s^k \int_{\R^n}
\delta\left(s-\omega\cdot y+\frac{|y|^2}{2}z\right)(\omega-zy)^\alpha \varphi(y) dy.
\label{DeriG}
\end{equation}

We define
$$
{\mathcal H}_m[\varphi](\sigma, \omega, z):=\frac{1}{2(2\pi)^{(m-1)/2}}\int_{\R} \chi_+^{(1-m)/2}\left(s-\sigma+
\frac{\sigma^2}{2}z\right){\mathcal G}[\varphi](s, \omega, z)ds.
$$
Then we obtain
\begin{equation}
\label{FreeSolExp02}
r^{(m-1)/2} \bigl(E_m(t)*\varphi\bigr)(x)={\mathcal H}_m[\varphi](\sigma, \omega, r^{-1}).
\end{equation}
Since we have
\begin{align*}
& {\mathcal H}_m[\varphi](\sigma, \omega, z)\\
& \quad =\frac{(-1)^k}{2(2\pi)^{(m-1)/2}}\int_{\R} \chi_+^{k+(1-m)/2}\left(s-\sigma+
\frac{\sigma^2}{2}z\right)\bigl(\pa_s^k{\mathcal G}[\varphi]\bigr)(s, \omega, z)ds
\end{align*}
for any nonnegative integer $k$, and since we have $\chi_+^a\in C^1(\R)$ for $a>1$,
we can easily see that
${\mathcal H}_m[\varphi]\in C^\infty(\R\times \Sp^{n-1}\times [0, (2M)^{-1}])$.
Moreover we have
\begin{equation}
\label{SupportHE}
\sigma \le M \text{ in $\supp {\mathcal H}_m[\varphi]$ when $m$ is even,}
\end{equation}
and
\begin{equation}
\label{SupportHO}
|\sigma| \le M\text{ in $\supp {\mathcal H}_m[\varphi]$ when $m(\ge 3)$ is odd.}
\end{equation}
Indeed \eqref{SupportHE} and \eqref{SupportHO} for $z\ne 0$
follow from \eqref{Huygens01'}, \eqref{Huygens02'}, and \eqref{FreeSolExp02}, 
while they follow immediately from the definition of ${\mathcal H}_m$ when $z=0$.
\subsection{The Radon transform and the Friedlander radiation field}
Let ${\mathcal S}(\R^n)$ denote the set of rapidly decreasing functions on $\R^n$.
For $\varphi\in {\mathcal S}(\R^n)$ we define the {\it Radon transform} ${\mathcal R}[\varphi]$ of $\varphi$ by
\begin{equation}
\label{DefRadonT}
{\mathcal R}[\varphi](\sigma, \omega):=\int_{y\cdot \omega=\sigma} \varphi(y) dS(y),
\quad (\sigma, \omega)\in \R\times \Sp^{n-1},
\end{equation}
where $dS(y)$ denotes the surface element on the hyperplane $\{y\in \R^n; y\cdot\omega=\sigma\}$.
It is easy to see that ${\mathcal R}[\varphi]\in{\mathcal S}(\R\times \Sp^{n-1})$.
Since we have 
${\mathcal R}[\varphi](\sigma, \omega)={\mathcal G}[\varphi](\sigma, \omega, 0)$,
it follows from \eqref{DeriG} that
\begin{equation}
\label{RadonDeriExp01}
{\mathcal R}[\pa_x^\alpha \varphi](\sigma, \omega)=\omega^\alpha \pa_\sigma^k{\mathcal R}[\varphi](\sigma, \omega)
\end{equation}
for any multi-index $\alpha$ with $|\alpha|=k\le 1$.


For a positive integer $m$ and $\varphi\in {\mathcal S}(\R^n)$, we put
\begin{align*}
{\mathcal R}_m[\varphi](\sigma, \omega):= & \frac{1}{2(2\pi)^{(m-1)/2}}\int_{\R}\chi_+^{(1-m)/2}(s-\sigma)
{\mathcal R}[\varphi](s,\omega) ds\\
=& \frac{1}{2(2\pi)^{(m-1)/2}}
\bigl(\chi_-^{(1-m)/2}*{\mathcal R}[\varphi](\cdot, \omega)\bigr)(\sigma),
\quad (\sigma, \omega)\in \R\times \Sp^{n-1},
\end{align*}
where $\chi_-^a(\sigma):=\chi_+^a(-\sigma)$ for $a\in \C$, and 
$*$ is the convolution with respect to $\sigma$-variable.
Note that we have ${\mathcal R}_m[\varphi]={\mathcal H}_m(\cdot,\cdot,0)\in C^\infty(\R\times \Sp^{n-1})$.
For any multi-index $\alpha$ with $|\alpha|=k\le 1$, we obtain
from \eqref{RadonDeriExp01} that
\begin{align}
{\mathcal R}_m[\pa_x^\alpha\varphi](\sigma,\omega)
=& \frac{1}{2(2\pi)^{(m-1)/2}}\int_{\R}\chi_+^{(1-m)/2}(s-\sigma)\omega^\alpha
\pa_s^k{\mathcal R}[\varphi](s,\omega)ds \nonumber\\
=& \frac{(-1)^k\omega^\alpha}{2(2\pi)^{(m-1)/2}}\int_{\R}\chi_+^{(1-m)/2-k}(s-\sigma){\mathcal R}[\varphi](s,\omega)ds\nonumber\\
=& \omega^\alpha \pa_\sigma^k{\mathcal R}_m[\varphi](\sigma, \omega).
\label{RadonDeriExp02}
\end{align}
If $\varphi(x)=0$ for $|x|\ge M$, then we immediately see by \eqref{DefRadonT} that 
\begin{equation}
\label{SupportRadon}
{\mathcal R}[\varphi](\sigma,\omega)=0, \quad |\sigma|\ge M, \ \omega\in \Sp^{n-1}.
\end{equation}
Consequently we get
\begin{equation}
\label{SupportRNG}
{\mathcal R}_m[\varphi](\sigma,\omega)=0, \quad \sigma\ge M,\ \omega\in \Sp^{n-1}.
\end{equation}
When $m(\ge 3)$ is odd, since $\supp \chi_+^{(1-m)/2}=\{0\}$, we obtain
\begin{equation}
\label{SupportRNGO}
{\mathcal R_m}[\varphi](\sigma,\omega)=0, \quad |\sigma|\ge M,\ \omega\in \Sp^{n-1}.
\end{equation}
\begin{lemma}\label{DecayPropRN}
Let $m$ be a positive integer with $m\ge 2$.
For $\varphi\in C^\infty_0(\R^n)$, a nonnegative integer $j$,
and a multi-index $\alpha$ with $|\alpha|=k\le 1$, 
there is a positive constant $C$ such that
$$
\left|\pa_\sigma^j{\mathcal R}_m[\pa_x^\alpha \varphi](\sigma, \omega)\right|\le
C \jb{\sigma}^{-j-k+(1-m)/2},\quad (\sigma, \omega)\in \R\times \Sp^{n-1}.
$$
\end{lemma}
\begin{proof}
Suppose that $\varphi(x)=0$ for $|x|\ge M$ with a positive constant $M$. 
Since $\pa_\sigma^j{\mathcal R}_m[\pa_x^\alpha \varphi]\in C^\infty(\R\times \Sp^{n-1})$,
we get
\begin{equation}
\label{DP01}
|\pa_\sigma^j{\mathcal R}_m[\pa_x^\alpha \varphi](\sigma, \omega)|\le 
C\le C\jb{2M}^{j+k+(m-1)/2}\jb{\sigma}^{-j-k+(1-m)/2}
\end{equation}
for $(\sigma,\omega)\in [-2M,M]\times \Sp^{n-1}$,
which implies the desired result for odd $m$ because of \eqref{SupportRNGO}.

Let $m$ be even. In view of \eqref{SupportRNG} and \eqref{DP01}, it suffices to
consider the case where $\sigma\le -2M$. 
Then we have $s-\sigma\ge |\sigma|/2\ge C\jb{\sigma}>0$ for $|s|\le M$.
Hence we obtain from \eqref{RadonDeriExp02} that
\begin{align*}
\left|\pa_\sigma^j{\mathcal R}_m[\pa_x^\alpha \varphi](\sigma, \omega)\right|=&\left|\frac{(-1)^{j+k}\omega^\alpha}{2(2\pi)^{(m-1)/2}}\int_{-M}^M \chi_+^{-j-k+(1-m)/2}(s-\sigma){\mathcal R}[\varphi](s, \omega) ds\right|\\
\le & C \int_{-M}^M (s-\sigma)^{-j-k+(1-m)/2} ds\le C\jb{\sigma}^{-j-k+(1-m)/2},
\end{align*}
because of \eqref{SupportRadon}. This completes the proof.
\end{proof}

\begin{lemma} \label{FriedAsymp00}
Let $n\ge 2$, and $m$ be a positive integer with $m\ge 2$. Suppose that $\varphi\in C^\infty_0(\R^n)$ and
$\varphi(x)=0$ for $|x|\ge M$ with a positive constant $M$.
Then, for any integer $j$ with $0\le j\le 2$, and a multi-index $\alpha$ with $|\alpha|=k\le 1$,
there is a positive constant $C$ such that
\begin{align}
& \left|r^{(m-1)/2}\pa_t^j\bigl(E_m(t)*\pa_x^\alpha\varphi\bigr)(x)-\bigl((-\pa_\sigma)^j{\mathcal R}_m[\pa_x^\alpha \varphi]\bigr)(r-t, \omega) \right| \nonumber\\
& \qquad\qquad\qquad\qquad\qquad
\le Cr^{-1}\jb{t-r}^{-j-k+(3-m)/2},\quad r\ge \frac{t}{2}\ge 2M,
\label{BFE01}
\end{align}
where $r=|x|$ and $\omega=|x|^{-1}x$.
\end{lemma}
\begin{proof}
Recall the definitions of ${\mathcal G}$ and ${\mathcal H}_m$ in the previous subsection,
and that we have
\begin{align*}
{\mathcal G}[\pa_x^\alpha \varphi](s, \omega,0)=&{\mathcal R}[\pa_x^\alpha \varphi](s,\omega),\ {\mathcal H}_m[\pa_x^\alpha \varphi](\sigma, \omega, 0)={\mathcal R}_m[\pa_x^\alpha \varphi](\sigma,\omega).
\end{align*}
We suppose that $r\ge t/2\ge 2M$. By \eqref{Huygens01'} and \eqref{SupportRNG}, we may also assume $r\le t+M$.

First we assume that $j=0$. 
We put $\sigma=r-t$ as before.
By \eqref{FreeSolExp02} we obtain
\begin{align*}
& \left|r^{(m-1)/2}\bigl(E_m(t)*\pa_x^\alpha\varphi\bigr)(x)-({\mathcal R}_m[\pa_x^\alpha \varphi])(r-t, \omega) \right|\\ 
& \quad =|{\mathcal H}_m[\pa_x^\alpha\varphi](\sigma,\omega, r^{-1})-{\mathcal H}_m[\pa_x^\alpha\varphi](\sigma, \omega, 0)|
\\
& \quad \le 
r^{-1}\int_0^1\left|(\pa_z {\mathcal H}_m[\pa_x^\alpha \varphi])(\sigma, \omega, \theta r^{-1})\right|d\theta,
\end{align*}
which leads to \eqref{BFE01} with $j=0$ if we can show
\begin{equation}
 \label{STAR}
|\pa_z {\mathcal H}_m[\pa_x^\alpha \varphi](\sigma, \omega, z)| \le
C\jb{\sigma}^{-k+(3-m)/2}
\end{equation}
for $(\sigma, \omega, z)\in \R\times \Sp^{n-1}\times [0,(2M)^{-1}]$.
Since ${\mathcal H}_m[\pa_x^\alpha \varphi]\in C^\infty\bigl(\R\times \Sp^{n-1}\times [0, (2M)^{-1}]\bigr)$, we have
\begin{equation}
\label{BHE01}
|\pa_z {\mathcal H}_m[\pa_x^\alpha \varphi](\sigma, \omega, z)|\le C\le
C\jb{\sigma}^{-k+(3-m)/2}
\end{equation}
for $(\sigma,\omega, z)\in [-2M, M]\times \Sp^{n-1}\times [0, (2M)^{-1}]$,
which leads to \eqref{STAR} for odd $m$ because of \eqref{SupportHO}.
Let $m$ be even. In view of \eqref{SupportHE} and \eqref{BHE01},
it suffices to show \eqref{STAR} for $(\sigma,\omega,z)\in (-\infty, -2M]\times \Sp^{n-1}\times[0,(2M)^{-1}]$.
Suppose $\sigma\le -2M$, $\omega\in \Sp^{n-1}$, and $0\le z \le(2M)^{-1}$.
We compute
\begin{align*}
& 2(2\pi)^{(m-1)/2}\pa_z{\mathcal H}_m[\pa_x^\alpha \varphi](\sigma, \omega, z) \\
& \quad = \int_{\R} \frac{\sigma^2}{2}\chi_+^{-(1+m)/2}\left(s-\sigma+\frac{\sigma^2}{2}z\right){\mathcal G}[\pa_x^\alpha \varphi](s,\omega, z)ds\\
&\qquad {}+\int_{\R} \chi_+^{(1-m)/2}\left(s-\sigma+\frac{\sigma^2}{2}z\right)(\pa_z{\mathcal G}[\pa_x^\alpha \varphi])(s,\omega, z)ds=:I_1+I_2.
\end{align*}
Since $s-\sigma+\sigma^2z/2\ge -5M/4+|\sigma|\ge 3|\sigma|/8\ge C\jb{\sigma}>0$ for 
$\sigma\le -2M$, $z\ge 0$, and $s\ge -5M/4$,
recalling \eqref{DeriG} we obtain
\begin{align*}
|I_1|\le & \left|\int_{\R} \frac{\sigma^2}{2}\chi_+^{-k-(1+m)/2}\left(s-\sigma+\frac{\sigma^2}{2}z\right)\Phi_\alpha(s,\omega, z)ds\right|\\
\le & C|\sigma|^2\jb{\sigma}^{-k-(1+m)/2}\le C\jb{\sigma}^{-k+(3-m)/2},
\end{align*}
where
$$
\Phi_{\alpha}(s, \omega, z)=\int_{\R^n}
\delta\left(s-\omega\cdot y+\frac{|y|^2}{2}z\right)(\omega-zy)^\alpha\varphi(y)dy
$$
which is a $C^\infty$-function in $\R\times \Sp^{n-1}\times [0, (2M)^{-1}]$
with $-5M/4\le s\le M$ in $\supp \Phi_\alpha$.
Similarly we get
$|I_2| \le C\jb{\sigma}^{-k+(3-m)/2}$.
This completes the proof of \eqref{STAR} for even $m$, and \eqref{BFE01} with $j=0$ is established.

Note that by \eqref{BFE01} with $j=0$ and Lemma~\ref{DecayPropRN} we get
\begin{equation}
\label{DecaySol11}
\bigl|r^{(m-1)/2}\bigl(E_m(t)*\pa_x^\alpha \varphi\bigr)(x) \bigr|\le C\jb{\sigma}^{-k+(1-m)/2}.
\end{equation}
Direct calculations lead to
$$
\pa_t E_m(t,x)=2\pi t E_{m+2}(t,x),\ 
-\pa_\sigma{\mathcal R}_m[\pa_x^\alpha \varphi](\sigma, \omega)
=2\pi {\mathcal R}_{m+2}[\pa_x^\alpha\varphi](\sigma,\omega).
$$
Hence it follows from \eqref{DecaySol11} and \eqref{BFE01} with $j=0$ that
\begin{align*}
& \bigl|r^{(m-1)/2}\pa_t \bigl(E_m(t)*\pa_x^\alpha\varphi\bigr)(x)-(-\pa_\sigma){\mathcal R}_m[\pa_x^\alpha\varphi](\sigma,\omega)\bigr|\\
& \quad \le 2\pi \bigl|r^{(m+1)/2} \bigl(E_{m+2}(t)*\pa_x^\alpha\varphi\bigr)(x)-{\mathcal R}_{m+2}[\pa_x^\alpha\varphi](\sigma,\omega)\bigr|\\
& \qquad {}+2\pi \bigl|r^{(m-1)/2}(t-r)\bigl(E_{m+2}(t)*\pa_x^\alpha \varphi\bigr)(x)\bigr|
\le Cr^{-1}\jb{\sigma}^{-k+(1-m)/2},
\end{align*}
which is \eqref{BFE01} for $j=1$. Observing that
$$
\pa_t^2E_m(t,x)=2\pi E_{m+2}(t,x)+(2\pi t)^2E_{m+4}(t)
$$
and
$$
(-\pa_\sigma)^2{\mathcal R}_m[\pa_x^\alpha \varphi](\sigma, \omega)
=(2\pi)^2 {\mathcal R}_{m+4}[\pa_x^\alpha \varphi](\sigma,\omega),
$$
we can show \eqref{BFE01} for $j=2$
in a similar way. 
\end{proof}

For $(\varphi, \psi)\in \bigl({\mathcal S}(\R^n)\bigr)^2$, we define
the {\it Friedlander radiation field} ${\mathcal F}_0[\varphi,\psi]$ by
\begin{equation}\label{DefGeneralFried}
{\mathcal F}_0[\varphi,\psi](\sigma, \omega)=-\pa_\sigma{\mathcal R}_n[\varphi](\sigma, \omega)+{\mathcal R}_n[\psi](\sigma, \omega), \quad (\sigma, \omega)\in \R\times \Sp^{n-1}.
\end{equation}
Observe that this definition is a generalization of the previous definition
\eqref{Friedlander3D} that was given only for $n=3$ and $(\varphi,\psi)\in \bigl(C^\infty_0(\R^3)\bigr)^2$, because ${\mathcal R}_3[h]=(4\pi)^{-1}{\mathcal R}[h]$.
The next lemma is a slight refinement of \cite[Theorem~6.2.1]{Hoe97}
(see also Friedlander \cite{Fri62} and Katayama-Kubo~\cite{Kat-Kub09a}).
\begin{lemma} \label{PointwiseFriedlander}
Let $n\ge 2$ and $c>0$. Let 
$w$ be the solution to the Cauchy problem \eqref{FreeWaveEq}--\eqref{FreeData}.
If $(w_0, w_1)\in \bigl(C^\infty_0(\R^n)\bigr)^2$, then
there is a positive constant $C$ such that
\begin{equation}
\bigl|r^{(n-1)/2} w(t, x)-W(r-ct,\omega)\bigr|
\le C\jb{t+r}^{-1}\jb{ct-r}^{(3-n)/2}
\label{FA01}
\end{equation}
and
\begin{equation}
\bigl|r^{(n-1)/2}\pa w(t,x)-\vec{\omega}_c(x)
(\pa_\sigma W)(r-ct,\omega)\bigr| 
 \le C\jb{t+r}^{-1}\jb{ct-r}^{(1-n)/2}
\label{FA02}
\end{equation}
for all $(t, x)\in [0,\infty)\times (\R^n\setminus\{0\})$,
where
$W(\sigma,\omega)={\mathcal F}_0[w_0, c^{-1}w_1](\sigma,\omega)$,
$\vec{\omega}_c(x)$ is given by \eqref{speedconst},
$r=|x|$, $\omega=|x|^{-1}x$, and $\pa=(\pa_0,\pa_1,\ldots, \pa_n)$.
\end{lemma}
\begin{proof}
We may assume that $c=1$, because the general result is easily obtained by 
a change of variables.
We assume that $w_0(x)=w_1(x)=0$ for $|x|\ge M$.

Firstly we suppose that $r\ge t/2\ge 2M$. Then we have $r^{-1}\le C\jb{t+r}^{-1}$.
Let $|\alpha|=k$, and suppose that $j$ and $k$ are nonnegative integers with 
$0\le j+k\le 1$.
From \eqref{FreeSolExp} we get
$$
\pa_t^j\pa_x^\alpha w(t,x)=\pa_t^{j+1}\bigl(E_n(t)*\pa_x^\alpha w_0\bigr)(x)+
\pa_t^j\bigl(E_n(t)*\pa_x^\alpha w_1\bigr)(x).
$$
Then Lemma~\ref{FriedAsymp00} (with $m=n$) and \eqref{RadonDeriExp02} lead to
\begin{align*}
r^{(n-1)/2}\pa_t^j\pa_x^\alpha w(t,x)=& 
(-1)^j\omega^{\alpha} \bigl(\pa_\sigma^{j+k}W\bigr)(r-t,\omega)
\\
&{}+O\left(\jb{t+r}^{-1}\jb{t-r}^{-j-k+(3-n)/2}\right),
\end{align*}
which implies \eqref{FA01} when $j+k=0$, and \eqref{FA02} when $j+k=1$.

Secondly we suppose that either $r\le t/2$ or $t\le 4M$ holds. 
Then we have $\jb{t-r}^{-1}\le C\jb {t+r}^{-1}$. Hence by Lemma~\ref{DecayPropRN}
(with $m=n$)
we get
\begin{equation}
\label{FaLast01}
|(\pa_\sigma^j W)(r-t, \omega)|\le C\jb{t-r}^{-j+(1-n)/2}\le C \jb{t+r}^{-j+(1-n)/2}.
\end{equation}
Now we are going to prove
\begin{equation}
\label{FaLast}
|r^{(n-1)/2}\pa^\alpha w(t,x)|\le C \jb{t+r}^{-|\alpha|+(1-n)/2}
\end{equation}
for $(t,x)$ with either $r\le t/2$ or $t\le 4M$, 
which implies \eqref{FA01} and \eqref{FA02}
with the help of \eqref{FaLast01}.
If we assume $-2M\le r-t\le M$ in addition, then we get $t\le 4M$ and $r\le 5M$.
Now, from \eqref{FreeSolExp} we can easily obtain
$$
|r^{(n-1)/2}\pa^\alpha w(t,x)|\le C \le C\jb{t+r}^{-|\alpha|+(1-n)/2},
\quad t\le 4M,\ r\le 5M.
$$
When $n$ is odd, this shows \eqref{FaLast} because of \eqref{Huygens01} and \eqref{Huygens02}. Thus we assume $n$ is even and $r-t\le -2M$. Accordingly we have
$t/2\ge r$, and \eqref{Huygens03} immediately implies the desired result.
This completes the proof.
\end{proof}
\subsection{The translation representation}
For $\varphi\in {\mathcal S}(\R^n)$, we write $\widehat{\varphi}(={\mathcal F}[\varphi])$ for its Fourier transform. To be more precise,
we put
$$
\widehat{\varphi}(\xi)={\mathcal F}[\varphi](\xi)=\frac{1}{(2\pi)^{n/2}}\int_{\R^n} e^{-i x\cdot\xi} \varphi(x) dx,\quad \xi\in \R^n,
$$
where $i=\sqrt{-1}$.
For a function $\psi=\psi(\sigma, \omega)\in {\mathcal S}(\R\times \Sp^{n-1})$, we define
$$
\widetilde{\psi}(\rho,\omega)=\widetilde{\mathcal F}[\psi](\rho,\omega):=\frac{1}{\sqrt{2\pi}}\int_{-\infty}^\infty e^{-i\rho\sigma}\psi(\sigma, \omega) d\sigma,\quad (\rho,\omega)\in \R\times \Sp^{n-1},
$$
which is the one-dimensional Fourier transform of $\psi(\cdot,\omega)$ with a parameter $\omega\in \Sp^{n-1}$.
Let $\varphi\in {\mathcal S}(\R^n)$. For $(\rho,\omega)\in \R\times \Sp^{n-1}$, we have
\begin{align}
\widehat{\varphi}(\rho\omega)=& \frac{1}{(2\pi)^{n/2}} \int_{\R^n} e^{-i\rho(y\cdot\omega)} \varphi(y) dy
=\frac{1}{(2\pi)^{n/2}}\int_{-\infty}^\infty e^{-i \rho \sigma} {\mathcal R}[\varphi](\sigma, \omega)d\sigma\nonumber\\
= &\frac{1}{(2\pi)^{(n-1)/2}} \widetilde{{\mathcal R}[\varphi]}(\rho,\omega).
\label{RadonFourier}
\end{align}
Equation \eqref{RadonFourier} implies
\begin{equation}
\label{Symm02}
\widetilde{{\mathcal R}[\varphi]}( -\rho, -\omega)=(2\pi)^{(n-1)/2}\widehat{\varphi}(\rho\omega)=\widetilde{{\mathcal R}[\varphi]}(\rho, \omega),\quad (\rho,\omega)\in \R\times \Sp^{n-1}.
\end{equation}
In other words $\widetilde{{\mathcal R}[\varphi]}$ is an even function in $(\rho,\omega)$.

\begin{lemma}\label{isometry}
For $(\varphi,\psi)\in \bigl({\mathcal S}(\R^n)\bigr)^2$ we have
$$
\|\pa_\sigma{\mathcal F}_0[\varphi,\psi]\|_{L^2(\R\times \Sp^{n-1})}=\|(\varphi,\psi)\|_{H_0(\R^n)}.
$$
\end{lemma}
\begin{proof}
Since we have ${\mathcal R}_n[h](\sigma, \omega)=B_n\bigl(\chi_-^{(1-n)/2}*{\mathcal R}[h](\cdot,\omega)\bigr)(\sigma)$ for $h\in {\mathcal S}(\R^n)$
with $B_n=1/\left(2(2\pi)^{(n-1)/2}\right)$, we get
\begin{equation}
\label{RadonFried}
\widetilde{{\mathcal R}_n[h]}(\rho,\omega)
={\mathcal F}_1\bigl[{\mathcal R}_n[h](\cdot,\omega)\bigr](\rho)
=\sqrt{2\pi}B_n \mathcal{F}_1\bigl[\chi_{-}^{(1-n)/2}\bigr](\rho)
\widetilde{{\mathcal R}[h]}(\rho,\omega),
\end{equation}
where ${\mathcal F}_1$ denotes the one-dimensional Fourier transformation (of a tempered distribution).
It is known that
$$
\mathcal{F}_1[\chi_-^a](\rho)=\frac{1}{\sqrt{2\pi}} e^{i\pi(a+1)/2}(\rho+i0)^{-a-1}
$$
for every $a\in \C$, where $z^b$ for $z\in \C\setminus (-\infty,0]$ and $b\in \C$ is given by
$z^b=\exp(b\log z)$ if we write $\log z=\log |z|+i \arg z$ with $-\pi<\arg z<\pi$,
and $(\rho+i0)^b$ is defined by $(\rho+i0)^{b}=\lim_{\ve \to +0} (\rho+i\ve)^{b}$
(see H\"ormander \cite[Example 7.1.17]{Hoe90-01} for instance).
Especially we have
$$
\bigl|\mathcal{F}_1\bigl[\chi_-^{(1-n)/2}\bigr](\rho)\bigr|^2=\frac{1}{2\pi} |\rho|^{n-3},\quad \rho\in \R.
$$
Therefore, for $\varphi, \psi \in {\mathcal S}(\R^n)$ 
and nonnegative integers $j, k$, it
follows from the Plancherel formula for the one-dimensional Fourier transform 
and \eqref{RadonFried} that
\begin{align}
& \left\langle \pa_\sigma^j {\mathcal R}_n[\varphi], \pa_\sigma^k {\mathcal R}_n[\psi]\right\rangle_{L^2(\R\times \Sp^{n-1})}
\nonumber\\
& \ ={(-i)^ki^{j}B_n^2}\int_{\Sp^{n-1}}\left(\int_{-\infty}^\infty 
\widetilde{{\mathcal R}[\varphi]}(\rho,\omega) \overline{\widetilde{{\mathcal R}[\psi]}(\rho,\omega)}
|\rho|^{n-3}\rho^{j+k} d\rho\right) dS_\omega,
\label{Pla02}
\end{align}
where $dS_\omega$ denotes the surface element on $\Sp^{n-1}$.
By \eqref{Symm02}, we see that the integrand on the right-hand side of \eqref{Pla02}
is an odd (resp.~even) function in $(\rho,\omega)$ if $j+k$ is odd (resp.~even).
Hence we have $\left\langle \pa_\sigma^2{\mathcal R}_n[\varphi], \pa_\sigma {\mathcal R}_n[\psi]\right\rangle_{L^2(\R\times \Sp^{n-1})}=0$.
Now we find
$$
\|\pa_\sigma{\mathcal F}_0[\varphi,\psi]\|_{L^2(\R\times \Sp^{n-1})}^2=\|\pa_\sigma^2{\mathcal R}_n[\varphi]\|_{L^2(\R\times \Sp^{n-1})}^2+\|\pa_\sigma{\mathcal R}_n[\psi]\|_{L^2(\R\times \Sp^{n-1})}^2.
$$
We obtain from \eqref{Pla02}, \eqref{Symm02}, and \eqref{RadonFourier} that
\begin{align*}
\|\pa_\sigma{\mathcal R}_n[\psi]\|_{L^2(\R\times \Sp^{n-1})}^2
=& 2{B_n^2}\int_{\Sp^{n-1}}\left(\int_{0}^\infty 
\left|\widetilde{{\mathcal R}[\psi]}(\rho,\omega)\right|^2 
\rho^{n-1} d\rho\right) dS_\omega
\\
=& \frac{1}{2}\bigl\|\widehat{\psi}\bigr\|_{L^2(\R^n)}^2=\frac{1}{2}\|\psi\|_{L^2(\R^n)}^2.
\end{align*}
Just in the same manner, we obtain
$$
\|\pa_\sigma^2{\mathcal R}_n[\varphi]\|_{L^2(\R\times \Sp^{n-1})}^2
=\frac{1}{2}\bigl\|\,|\cdot|\,\widehat{\varphi}\bigr\|_{L^2(\R^n)}^2=\frac{1}{2}\|\nabla_x \varphi\|_{L^2(\R^n)}^2.
$$
This completes the proof.
\end{proof}

By Lemma~\ref{isometry}, we can uniquely extend the linear mapping
$$
\bigl({\mathcal S}(\R^n)\bigr)^2\ni (\varphi, \psi) \mapsto \pa_\sigma{\mathcal F}_0[\varphi,\psi]
\in L^2(\R\times \Sp^{n-1})
$$
to the linear mapping ${\mathcal T}$ from $H_0(\R^n)$ to $L^2(\R^n)$
with
\begin{equation}
\|{\mathcal T}[\varphi,\psi]\|_{L^2(\R\times \Sp^{n-1})}=\|(\varphi, \psi)\|_{H_0(\R^n)},\quad (\varphi,\psi)\in H_0(\R^n).
\label{Isos}
\end{equation}
This mapping ${\mathcal T}$ is called the {\it translation representation}
in Lax-Phillips~\cite{LaxPhi89}.

The following lemma was essentially proved in \cite{LaxPhi89}
for odd $n$, and in \cite{LaxPhi73} for even $n$
(see also Melrose~\cite{Mel79}). 
\begin{lemma}
The mapping $\mathcal T$ defined above is an isometric isomorphism from $H_0(\R^n)$ to
$L^2(\R\times \Sp^{n-1})$.
\end{lemma}
\begin{proof}
What is left to show is that ${\mathcal T}$ is surjective.
Because of \eqref{Isos}, we only have to prove the following: For any function $v$ in
some dense subset of $L^2(\R\times \Sp^{n-1})$, there is $(\varphi,\psi)\in H_0(\R^n)$ such that ${\mathcal T}[\varphi,\psi]=v$. 

We write $h\in {\mathcal S}_0(\R\times \Sp^{n-1})$
if $h 
\in{\mathcal S}(\R\times \Sp^{n-1})$
and there is a positive constant $\delta$ such that 
$h(\rho, \omega)=0$ for all $(\rho, \omega)\in (-\delta,\delta)\times \Sp^{n-1}$.
We put
$$
{\mathcal S}_1(\R\times \Sp^{n-1})=\left\{
\widetilde{\mathcal F}^{-1}[h]; h\in {\mathcal S}_0(\R\times \Sp^{n-1})
                                  \right\},
$$
where 
$$
\widetilde{{\mathcal F}}^{-1}[h](\sigma,\omega)=\frac{1}{\sqrt{2\pi}}\int_{-\infty}^\infty e^{i\rho\sigma}
h(\rho, \omega) d\rho,
$$
which is the one-dimensional inverse Fourier transform of $h(\cdot,\omega)$ with a parameter $\omega\in \Sp^{n-1}$. 
It is easy to see that ${\mathcal S}_0(\R\times \Sp^{n-1})$ is dense in
$L^2(\R\times\Sp^{n-1})$, and hence ${\mathcal S}_1(\R\times \Sp^{n-1})$
is also dense in $L^2(\R\times\Sp^{n-1})$.

Let $v\in {\mathcal S}_1(\R\times \Sp^{n-1})$. 
We want to find $(\varphi, \psi)\in H_0(\R^n)$ such that
${\mathcal T}[\varphi,\psi]=v$.
We put
$$
v_0(\rho, \omega)=\frac{1}{\sqrt{2\pi}B_n{\mathcal F}_1\bigl[\chi_-^{(1-n)/2}\bigr](\rho)}
\widetilde{v}(\rho,\omega),\quad (\rho,\omega)\in \R\times \Sp^{n-1}.
$$
Note that we have $v_0\in {\mathcal S_0}(\R\times \Sp^{n-1})$, because
$\widetilde{v}\in \mathcal{S}_0(\R\times \Sp^{n-1})$, and the singularity of 
$1/{\mathcal F}_1\bigl[\chi_-^{(1-n)/2}\bigr]$ lies only at $\rho=0$.
For $\xi\in \R^n\setminus\{0\}$ we put
\begin{align*}
v_1(\xi)=& \frac{v_0(|\xi|, |\xi|^{-1}\xi)+v_0(-|\xi|, -|\xi|^{-1}\xi)}{2(2\pi)^{(n-1)/2}|\xi|^2}, \\
v_2(\xi)=& \frac{v_0(|\xi|, |\xi|^{-1}\xi)-v_0(-|\xi|, -|\xi|^{-1}\xi)}{2i(2\pi)^{(n-1)/2}|\xi|}.
\end{align*}
We also set $v_1(0)=v_2(0)=0$. Then, since $v_0\in{\mathcal S}_0(\R\times \Sp^{n-1})$,
we find that $v_1, v_2\in {\mathcal S}(\R^n)$. Hence 
if we set $\varphi={\mathcal F}^{-1}[v_1]$ and $\psi={\mathcal F}^{-1}[v_2]$,
then we get $(\varphi,\psi)\in \bigl({\mathcal S}(\R^n)\bigr)^2\subset H_0(\R^n)$.
Using \eqref{RadonFourier} and \eqref{RadonFried}, we get
$$
\widetilde{\mathcal F}\left[ \pa_\sigma{\mathcal F}_0[\varphi,\psi]\right](\rho,\omega)=\rho^2
\widetilde{{\mathcal R}_n[\varphi]}(\rho,\omega)+i\rho\widetilde{\mathcal{R}_n[\psi]}(\rho,\omega)=\widetilde{v}(\rho,\omega)
$$
for $(\rho, \omega)\in \R\times \Sp^{n-1}$,
which shows ${\mathcal T}[\varphi,\psi]=\pa_\sigma{\mathcal F}_0[\varphi, \psi]=v$.
This completes the proof.
\end{proof}

Theorem~\ref{AFSE} is an immediate consequence of the following lemma.
\begin{lemma}\label{Final01} Let $c>0$, and $(w_0, w_1)\in H_0(\R^n)$.
Let $w$ be a solution to the Cauchy problem \eqref{FreeWaveEq}--\eqref{FreeData}, and
$$
{\mathcal W}(t,x)=|x|^{-(n-1)/2}{\mathcal T}[w_0, c^{-1}w_1](|x|-ct,|x|^{-1}x)
$$
for $(t,x)\in [0,\infty)\times (\R^n\setminus\{0\})$.
Then we have
$$
\lim_{t\to \infty} 
\bigl\|\pa^{(c)} w(t,\cdot)-\vec{\omega}(\cdot){\mathcal W}(t,\cdot)\bigr\|_{L^2(\R^n)} 
=0,
$$
where $\pa^{(c)}:=(c^{-1}\pa_t, \pa_1, \ldots , \pa_n)$
and $\vec{\omega}(x):=\vec{\omega}_1(x)=(-1, |x|^{-1}x)$.
\end{lemma}
\begin{proof}
Let $\ve >0$. There is $(w_0^*,w_1^*)\in \bigl(C^\infty_0(\R^n)\bigr)^2$ such that
$$
\left\|(w_0, c^{-1}w_1)-(w_0^*, c^{-1} w_1^*)\right\|_{H_0(\R^n)}^2<\left(\frac{\ve}{3}\right)^2.
$$
Let $w^*$ and ${\mathcal W}^*$ be defined similarly to $w$ and ${\mathcal W}$,
respectively, by replacing $(w_0, w_1)$
with $(w_0^*, w_1^*)$.
Recalling \eqref{DefNormH0}, we obtain from the energy identity that
\begin{align*}
& \frac{1}{2}\|\pa^{(c)} w(t,\cdot)-\pa^{(c)} w^*(t,\cdot)\|_{L^2(\R^n)}^2 \\
& \quad = \frac{1}{2}\left(\|c^{-1}(w_1-w_1^*)\|_{L^2(\R^n)}^2+\|\nabla_x(w_0-w_0^*)\|_{L^2(\R^n)}^2\right)\\
& \quad = \|(w_0, c^{-1}w_1)-(w_0^*, c^{-1}w_1^*)\|_{H_0(\R^n)}^2 < \left(\frac{\ve}{3}\right)^2.
\end{align*}
Since $|\vec{\omega}(x)|=2$, using \eqref{Isos} we get
\begin{align}
& \frac{1}{2}\bigl\|
 \vec{\omega}(\cdot)\bigl({\mathcal W}(t,\cdot)-{\mathcal W}^*(t,\cdot)\bigr)
\bigr\|_{L^2(\R^n)}^2
\nonumber\\
& \quad =
\int_0^\infty 
\left(\int_{\Sp^{n-1}} |{\mathcal T}[w_0-w_0^*, c^{-1}(w_1-w_1^*)](r-ct,\omega)|^2 dS_\omega\right) dr \nonumber\\
& \quad \le \|{\mathcal T}[w_0-w_0^*, c^{-1}(w_1-w_1^*)]\|_{L^2(\R\times \Sp^{n-1})}^2
\nonumber\\
& \quad =\|(w_0, c^{-1}w_1)-(w_0^*, c^{-1}w_1^*)\|_{H_0(\R^n)}^2< \left(\frac{\ve}{3}\right)^2.
\nonumber
\end{align}
By Lemma~\ref{PointwiseFriedlander} we obtain
\begin{align*}
& \frac{1}{2}\|\pa^{(c)} w^*(t,\cdot)-\vec{\omega}(\cdot){\mathcal W}^*(t,\cdot)\|_{L^2(\R^n)}^2
\\
& \quad \le \frac{1}{2} \int_0^\infty\left(\int_{\Sp^{n-1}} |r^{(n-1)/2} \pa^{(c)} w^{*}(t,r\omega)-
\vec{\omega}(r\omega) (\pa_\sigma {\mathcal F}_0^*) (r-ct,\omega)|^2dS_\omega\right)dr
\\
& \quad \le C \int_0^\infty (1+t+r)^{-2}(1+|t-r|)^{-(n-1)} dr\le C(1+t)^{-1},
\end{align*}
where
${\mathcal F}_0^*(\sigma, \omega)={\mathcal F}_0[w_0^*, c^{-1}w_1^*](\sigma, \omega)$.
Hence there is a positive constant $t_0>0$ such that $t\ge t_0$ implies
$$
\frac{1}{2} \|\pa^{(c)} w^*(t,\cdot)-\vec{\omega}(\cdot){\mathcal W}^*(t,\cdot)\|_{L^2(\R^n)}^2
<\left(\frac{\ve}{3}\right)^2.
$$
To sum up, for any $\ve>0$ there is a positive constant $t_0$ such that we have
$$
\frac{1}{\sqrt{2}} \bigl\|\pa^{(c)} w(t,\cdot)-\vec{\omega}(\cdot){\mathcal W}(t,\cdot)\bigr\|_{L^2(\R^n)} 
<\ve,\quad t\ge t_0.
$$
This completes the proof.
\end{proof}

Now we are in a position to prove Theorem~\ref{AFSE}.
\begin{proof}[Proof of Theorem~$\ref{AFSE}$]
Suppose that there is $V\in L^2(\R\times \Sp^{n-1})$ such that
\begin{equation}
\label{Fin03}
\lim_{t\to \infty} \bigl\|\pa v(t,\cdot)-\vec{\omega}_c(\cdot){V}^\sharp(t,\cdot)\bigr\|_{L^2(\R^n)} 
=0,
\end{equation}
where 
\begin{equation}
\label{Fin01a}
 {V}^\sharp(t,x)=|x|^{-(n-1)/2} V(|x|-ct, |x|^{-1}x).
\end{equation}
Then we get
\begin{equation}
\label{Fin03a}
\lim_{t\to \infty} 
\bigl\|\pa^{(c)} v(t,\cdot)-\vec{\omega}(\cdot){V}^\sharp(t,\cdot)\bigr\|_{L^2(\R^n)} 
=0,
\end{equation}
where $\pa^{(c)}$ and $\vec{\omega}(x)$ are defined as in Lemma~\ref{Final01}.
We define $(w_0, c^{-1}w_1)={\mathcal T}^{-1} [V]\bigl(\in H_0(\R^n)\bigr)$,
and let $w$ be the solution to the Cauchy problem \eqref{FreeWaveEq}--\eqref{FreeData}
for this $(w_0, w_1)$.
Then it follows from Lemma~\ref{Final01} that
\begin{equation}
\label{Fin02}
\lim_{t\to\infty} 
\bigl\|\pa^{(c)} w(t,\cdot)-\vec{\omega}(\cdot){V}^\sharp(t,\cdot)\bigr\|_{L^2(\R^n)} 
=0,
\end{equation}
which, together with \eqref{Fin03a} implies
\begin{equation}
\label{Fin01}
\lim_{t\to\infty} \|v(t,\cdot)-w(t,\cdot)\|_{E,c} 
=\lim_{t\to\infty}\frac{1}{\sqrt{2}}\|\pa^{(c)}v(t,\cdot)-\pa^{(c)}w(t,\cdot)\|_{L^2(\R^n)} 
=0.
\end{equation} 

Conversely, suppose that there is $(w_0, w_1)\in H_0(\R^n)$ such that
\eqref{Fin01} holds for the solution $w$ to the Cauchy problem \eqref{FreeWaveEq}--\eqref{FreeData}.
If we put 
$$
 V(\sigma, \omega)={\mathcal T}[w_0, c^{-1}w_1](\sigma, \omega)\bigl(\in L^2(\R\times \Sp^{n-1})\bigr)
$$
and define ${V}^\sharp$ by \eqref{Fin01a},
then Lemma~\ref{Final01} implies \eqref{Fin02}. 
Now \eqref{Fin01} and \eqref{Fin02} yield \eqref{Fin03a}, which is equivalent to \eqref{Fin03}.
This completes the proof.
\end{proof}
\section{Vector Fields Associated with the Wave Equations}\label{P00}
We restrict our attention to the three space dimensional case from now on.
We introduce vector fields
\begin{align*}
S: = & t\pa_t+x\cdot\nabla_x,\\
\Omega= & (\Omega_1,\Omega_2,\Omega_3):=x\times \nabla_x=(x_2\pa_3-x_3\pa_2, x_3\pa_1-x_1\pa_3, x_1\pa_2-x_2\pa_1),
\end{align*}
where the symbols $\cdot$ and $\times$ denote the inner and exterior products in $\R^3$, respectively.
We also use $\pa=(\pa_t,\nabla_x)=(\pa_a)_{0\le a\le 3}$. We define
$$
Z=(Z_0, Z_1,\ldots, Z_7)=(S,\Omega,\pa)=(S, \Omega_1,\Omega_2, \Omega_3,\pa_0,\pa_1,\pa_2,\pa_3).
$$
These vector fields are compatible with the system of wave equations having
 multiple propagation speeds, since we have
$[\dal_c, \Omega_j]=[\dal_c, \pa_a]=0$ for $1\le j\le 3$ and $0\le a\le 3$, and $[\dal_c, S]=2\dal_c$,
where $c>0$, and $[P,Q]=PQ-QP$ for operators $P$ and $Q$. Using a multi-index $\alpha=(\alpha_0,\alpha_1,\ldots, \alpha_7)$,
we write
$Z^\alpha=Z_0^{\alpha_0}Z_1^{\alpha_1}\cdots Z_7^{\alpha_7}$.
For a nonnegative integer $s$ and a (scalar- or vector-valued) smooth function $\varphi=\varphi(t,x)$, we define
$$
|\varphi(t,x)|_s=\sum_{|\alpha|\le s}| Z^\alpha\varphi(t,x)|.
$$
We can check that we have $[Z_a, Z_b]=\sum_{d=0}^7 C^{ab}_d Z_d$ and $[Z_a, \pa_b]=\sum_{d=0}^3 D^{ab}_d \pa_d$
with appropriate constants $C^{ab}_d$ and $D^{ab}_d$. Hence,
for any multi-indices $\alpha$ and $\beta$, and any nonnegative integer $s$,
there exist some positive constants $C_{\alpha,\beta}$ and $C_s$ such that we have
\begin{align*}
& |Z^\alpha Z^\beta \varphi(t,x)|\le C_{\alpha,\beta}|\varphi(t,x)|_{|\alpha|+|\beta|},\\
& C_s^{-1}|\pa \varphi(t,x)|_s \le \sum_{|\alpha|\le s}\sum_{a=0}^3 |\pa_a Z^\alpha\varphi(t,x)|
\le C_s|\pa\varphi(t,x)|_s
\end{align*}
for any smooth function $\varphi$.

We write $r=|x|$, $\omega=(\omega_1,\omega_2,\omega_3)=|x|^{-1}x$, and $\pa_r=\sum_{j=1}^3\omega_j \pa_j$.
Then we can write $S=t\pa_t+r\pa_r$.
For $c> 0$, we define
$$
\pa_{\pm}^{(c)}=\pa_t\pm c\pa_r,\text{ and } D_\pm^{(c)}=\pm\frac{1}{2c}\pa_\pm^{(c)}=\frac{1}{2}(\pa_r\pm c^{-1}\pa_t).
$$
Since we have $\pa_t=-c\bigl(D_-^{(c)}-D_+^{(c)}\bigr)$, $\pa_r=D_-^{(c)}+D_{+}^{(c)}$, and 
$$
 (1+r)\pa_{+}^{(c)}=cS-(ct-r)\pa_t+(\pa_t+c\pa_r),
$$ 
there exists a positive constant $C$ such that
\begin{align}
&
\bigl|\bigl(\pa_t-(-c)D_{-}^{(c)}\bigr)\varphi(t,x)\bigr|+\bigl|\bigl(\pa_r-D_-^{(c)}\bigr)
\varphi(t,x)\bigr| \nonumber\\
& \qquad\qquad\quad
\le \frac{C}{1+r}\left(|Z\varphi(t,x)|+|ct-r|\,|\pa\varphi(t,x)|\right)
\label{frame01}
\end{align}
for any smooth function $\varphi$. Since we have
$\nabla_x=\omega\pa_r-r^{-1}\omega\times \Omega$,
we get
\begin{equation}
\label{frame02}
|(\pa_k-\omega_k\pa_r)\varphi(t,x)|\le C(1+r)^{-1}|Z\varphi(t,x)|,\quad k=1,2,3.
\end{equation}
From \eqref{frame01} and \eqref{frame02} we obtain
\begin{equation}
\label{frame03}
\left|\pa \varphi(t,x)-\vec{\omega}_c(x)D_-^{(c)}\varphi(t,x)\right|
\le \frac{C}{1+r}\left(|Z\varphi(t,x)|+|ct-r|\,|\pa\varphi(t,x)|\right)
\end{equation}
for any smooth function $\varphi$, where $\vec{\omega}_c(x)$ 
is given by \eqref{speedconst}. As an immediate consequence we also get
\begin{equation}
\label{frame04}
\left|r\pa \varphi(t,x)-\vec{\omega}_c(x)D_{-}^{(c)}\bigl(r\varphi(t,x)\bigr)\right|
\le C\left(|\varphi(t,x)|_1+|ct-r|\,|\pa\varphi(t,x)|\right).
\end{equation}
We remark that the term $|ct-r|\,|\pa \varphi|$ was not needed
in the estimates used in \cite{Kat11} instead of \eqref{frame03} and \eqref{frame04}; 
however the vector field $L=t\nabla_x+x\pa_t$, which is not compatible with the multiple speed case, was involved.
We need the term $|ct-r|\,|\pa \varphi|$ here to compensate the lack of the vector field $L$.
These kinds of identities and estimates without the vector field $L$ were developed 
and used in \cite{KlaSid96}, \cite{Age00}, \cite{Sid00}, \cite{SidTu01}, \cite{Yok00}, and so on
(see also \cite{Kat12}, \cite{KatKub08}, \cite{SidThomases05}, \cite{SidThomases06}, 
and \cite{SidThomases07} for the related topics).


Using \eqref{frame03}, we can easily show the following estimate for the null forms $Q_0$ and $Q_{ab}$ given by \eqref{NullForm01} and \eqref{NullForm02}
(see \cite{Sid00}, \cite{SidTu01} and \cite{Yok00} for the details of the proof):
\begin{lemma}\label{NullEst}
Let $c>0$. Then we have
\begin{align*}
|Q_0(\varphi, \psi;c)|+\sum_{0\le a<b\le 3}
|Q_{ab}(\varphi,\psi)|\le & C\jb{r}^{-1}\left(|Z\varphi|\,|\pa\psi|
{}+|\pa\varphi|\,|Z\psi|\right)\\
&{}+C\jb{r}^{-1}\jb{ct-r}|\pa\varphi|\,|\pa\psi|
\end{align*}
at $(t,x)\in(0,\infty)\times \R^3$ with $r=|x|$.
\end{lemma}
\begin{proof}[Outline of proof] Fix $c>0$, and we define $R=(R_a)_{0\le a\le 3}=\pa-\vec{\omega}_c(x)D_-^{(c)}$.
Substituting $\pa=\vec{\omega}_c(x)D_{-}^{(c)}+R$, we get
\begin{align*}
  |Q_0(\varphi,\psi;c)|+\sum_{a,b}|Q_{ab}(\varphi,\psi)|\le & C 
\bigl(\bigl|D_-^{(c)}\varphi\bigr|\,|R\psi|+|R\varphi|\,\bigl|D_-^{(c)}\psi\bigr|+|R\varphi|\,|R\psi|\bigr)\\
  \le & C \left(|\pa\varphi|\,|R\psi|+|R\varphi|\,|\pa\psi|\right),
\end{align*}
where we have used $|R\psi|\le C|\pa \psi|$ to obtain the last line. 
The point here is that terms including $\bigl(D_-^{(c)}\varphi\bigr)\bigl(D_-^{(c)}\psi\bigr)$ are 
canceled out because of the structure of the null forms. Now, using \eqref{frame03}
to evaluate $|R\varphi|$ and $|R\psi|$, we obtain the desired result.
\end{proof}

Let $c_1, \ldots, c_N$ satisfy \eqref{distinct speeds}. We define
\begin{equation}
\label{DefLambda0}
\Lambda_0=\left\{(t,r)\in \R_+ \times \R_+; r\ge \frac{c_1t}{2}\ge 1\right\},
\end{equation}
where $\R_+=[0,\infty)$.
For $1\le j\le N$, we define
\begin{equation}
t_{0, j}(\sigma):=\max\{-2(2c_j-c_1)^{-1}\sigma, 2c_1^{-1}\}
\label{Deft0i}
\end{equation}
so that we have
$\bigcup_{\sigma\in \R}
\{(t, c_j t+\sigma); t\ge t_{0,j}(\sigma)\}=\Lambda_0$.
We also define $r_{0, j}(\sigma):=c_jt_{0, j}(\sigma)+\sigma$ for $\sigma\in \R$.
The following lemma is a modification of a key lemma in \cite{Kat11}
to obtain the asymptotic behavior.
\begin{lemma}\label{KeyLemma}
Assume \eqref{distinct speeds}, and let $j\in\{1,\ldots, N\}$.
Suppose the following:
\begin{itemize}
\item  $\mu_j>1$ and $\kappa_j\ge 0$.
\item $\mu_k\ge 0$ and $\kappa_k>1$ for any $k=1,\ldots, N$ with $k\ne j$.
\end{itemize}
If $v=v(t,r,\omega)\in C^1(\Lambda_0\times \Sp^2)$ satisfies
\begin{equation}
\label{Kanon}
\pa_+^{(c_j)}v(t,r,\omega)=G(t,r,\omega),\quad (t,r)\in \Lambda_0,\ \omega\in \Sp^2,
\end{equation}
and $G$ satisfies
\begin{equation}
\label{Air}
|G(t, r, \omega)|\le \sum_{k=1}^N B_k\jb{t+r}^{-\mu_k}\jb{c_kt-r}^{-\kappa_k},\quad (t,r)\in \Lambda_0,\ \omega\in \Sp^2
\end{equation}
with some positive constants $B_1,\ldots, B_N$, then there
exists a positive constant $C$ such that
\begin{align}
|v(t,r,\omega)-V(r-c_jt,\omega)|\le &
{CB_j}\jb{t+r}^{-\mu_j+1}\jb{c_jt-r}^{-\kappa_j}
\nonumber\\
& {}
+C\sum_{\substack{1\le k\le N \\ k\ne j}} B_k \jb{t+r}^{-\mu_k}
\label{Clannad}
\end{align}
for any $(t,r)\in \Lambda_0$ and $\omega\in \Sp^2$, where $V$ is defined by
\begin{equation}
\label{LittleBusters}
V(\sigma, \omega)=v\bigl(t_{0, j}(\sigma), r_{0, j}(\sigma),\omega \bigr)
{}+\int_{t_{0, j}(\sigma)}^\infty G(s, c_j s+\sigma, \omega) ds.
\end{equation}
The constant $C$ above is determined only by 
$c_k$, $\mu_k$, and $\kappa_k$ with $1\le k\le N$.
\end{lemma}
\begin{proof} 
First we note that $V$ is well-defined because of \eqref{Air}.
By \eqref{Kanon} we find
$$
v(t,c_jt+\sigma,\omega)=v\bigl(t_{0,j}(\sigma), r_{0,j}(\sigma),\omega\bigr)
{}+\int_{t_{0,j}(\sigma)}^t G(s, c_js+\sigma, \omega) ds,\quad t\ge t_{0,j}(\sigma)
$$
for $(\sigma,\omega) \in \R\times \Sp^2$. From \eqref{LittleBusters} and \eqref{Air} we get
\begin{equation}
|v(t,c_jt+\sigma,\omega)-V(\sigma, \omega)|
\le \int_{t}^\infty |G(s, c_js+\sigma, \omega)|ds
\le C\sum_{k=1}^N B_k I_k(t,\sigma),
\label{TomoyoAfter}
\end{equation}
where
$$
I_k(t,\sigma)=\int_t^\infty \left. (1+\tau+\rho)^{-\mu_k}(1+|c_k\tau-\rho|)^{-\kappa_k}\right|_{(\tau, \rho)=(s,c_js+\sigma)}ds
$$
for $1\le k\le N$.
By direct calculations, we obtain
\begin{equation}
\label{Planetarian}
I_j(t,\sigma)=\frac{1}{(c_j+1)(\mu_j-1)}(1+(c_j+1)t+\sigma)^{-\mu_j+1}(1+|\sigma|)^{-\kappa_j},
\end{equation}
since $\mu_j>1$.
If $k\ne j$, then we have
\begin{align}
I_k(t,\sigma)\le & (1+(c_j+1)t+\sigma)^{-\mu_k}\int_t^\infty 
(1+|(c_k-c_j)s-\sigma|)^{-\kappa_k}ds 
\nonumber\\
\le & (1+(c_j+1)t+\sigma)^{-\mu_k}\frac{1}{|c_k-c_j|}\int_{-\infty}^\infty(1+|\tau|)^{-\kappa_k}d\tau \nonumber\\
\le & \frac{2}{(\kappa_k-1)|c_k-c_j|}(1+(c_j+1)t+\sigma)^{-\mu_k}
\label{KudoWafter}
\end{align}
for $t\ge t_{0,j}(\sigma)$ and $\sigma\in \R$,
because we have $\mu_k\ge 0$, $\kappa_k>1$, and $c_k\ne c_j$.
Now we obtain \eqref{Clannad} immediately by putting $\sigma=r-c_jt$
in \eqref{TomoyoAfter}, \eqref{Planetarian}, and \eqref{KudoWafter}.
\end{proof}
%
\section{Proof of Theorem~\ref{Main01} and Corollary~\ref{AsympFreeEnergy}}\label{P01}
In the following, we write $r=|x|$ and $\omega=|x|^{-1}x$.
Let the assumptions in Theorem~\ref{Main01} be fulfilled.
Then we have the global solution $u$ to the Cauchy problem \eqref{OurSystem}--\eqref{OurData} for sufficiently small $\ve$
by the global existence theorems in \cite{SidTu01}, \cite{Sog03}, or \cite{Yok00}.
For $j=1, \ldots, N$, we put
$$
{u}_j^1(t,x)=u_j(t,x)-\ve u_j^0(t,x),
$$
where $u_j^0$ is the solution to $\dal_{c_j}u_j^0=0$ with initial data
$u_j^0=f_j$ and $\pa_t u_j^0=g_j$ at $t=0$. 
From \eqref{OurSystem} and \eqref{OurData}, we have
\begin{align}
\label{OurSysM}
& \dal_{c_j} u_j^1=F_j(\pa u, \pa^2 u)\quad\text{ in $(0,\infty)\times \R^3$}, \\
\label{OurDataM}
& u_j^1(0,x)=(\pa_t {u}_j^1)(0,x)=0,\quad x\in \R^3.
\end{align}
We put
$$
U_j^0(\sigma, \omega)={\mathcal F}_0[f_j, c_j^{-1}g_j](\sigma, \omega).
$$
Suppose that we have $f(x)=g(x)=0$ for $|x|\ge M$.
Then, as in \eqref{Huygens01}, \eqref{Huygens02}, and \eqref{SupportRNGO},
we get
\begin{align}
\label{Supportui0}
u_j^0(t,x)=&0,\quad |r-c_jt|\ge M,\\
U_j^0(\sigma,\omega)=&0,\quad |\sigma|\ge M
\label{SupportUi0}
\end{align}
for $1\le j\le N$.
Hence Lemma~\ref{PointwiseFriedlander} implies
\begin{align}
& \left| r u_j^0(t,x)-U_j^0(r-c_jt,x) \right|
{}+\left|
    r \pa u_j^0(t,x)-\vec{\omega}_{c_j}(x)(\pa_\sigma U_j^0)(r-c_jt,x)
   \right|
\nonumber\\
& \qquad\qquad\qquad\qquad\qquad\qquad
\le C\jb{t+r}^{-1}\jb{c_jt-r}^{-1},\quad 1\le j\le N
\label{FreeApprox}
\end{align}
for $(t,x)\in [0,\infty)\times(\R^3\setminus\{0\})$, because \eqref{Supportui0} and \eqref{SupportUi0} 
imply $1\le \jb{c_jt-r}\le \jb{M}$ on
the support of the functions on the left-hand side of \eqref{FreeApprox}.

Let $\Lambda$ 
be the set of $(t,x)\in [0,\infty)\times \R^3$ with $(t,r)\in \Lambda_0$,
where $\Lambda_0$ is given by \eqref{DefLambda0}; namely we put
$$
\Lambda=\left\{(t,x)\in[0,\infty)\times (\R^3\setminus\{0\}); |x|\ge \frac{c_1t}{2}\ge 1\right\}.
$$
We set $\Lambda^{\rm c}=\bigl([0,\infty)\times (\R^3\setminus\{0\})\bigr) \setminus \Lambda$. 
Recall the definitions of $\pa_\pm^{(c)}$ and $D_\pm^{(c)}$ in the previous section.
\begin{proof}[Proof of Theorem~$\ref{Main01}$ $(1)$]
Following the proof of the global existence theorem in
\cite{Kub-Yok01} (see also \cite[Proposition 4.2]{Kat04:02} and its proof) we obtain
\begin{align}
|u_k(t,x)|_2+\ve^{-1}|{u}_k^1(t,x)|_2\le & C\ve \jb{t+r}^{-1}\log \left(1+\frac{1+c_k t+r}{1+|c_k t-r|}\right),
\label{Basic01a}\\
|\pa {u}_k(t,x)|_1+\ve^{-1}|\pa {u}_k^1(t,x)|_1\le & 
C\ve \jb{r}^{-1}\jb{c_kt-r}^{-1}
\label{Basic01b}
\end{align}
for $1\le k\le N$.
We choose small $\delta>0$. Then, by \eqref{Basic01a} we get
\begin{equation}
\label{Basic01a'}
|{u}_k(t,x)|_2+\ve^{-1}|{u}_k^1(t,x)|_2\le C\ve \jb{t+r}^{-1+\delta}\jb{c_kt-r}^{-\delta}
\end{equation}
for $1\le k\le N$, because there is a positive constant $C_\delta$ such that
we have $\log(1+z)\le C_\delta z^\delta$ for $z\ge 1$.

Fix $j=1,2,\ldots, N$.
Switching to the polar coordinates, we get
\begin{equation}
\label{Polar}
r(\dal_{c_j} {u}_j^1)(t, r\omega)
=\pa_{+}^{(c_j)}\pa_{-}^{(c_j)}\bigl(r{u}_j^1(t, r\omega)\bigr)
{}-c_j^2 r^{-1}\Delta_\omega {u}_j^1(t,r\omega)
\end{equation}
for $(t,r,\omega)\in (0,\infty)\times (0,\infty)\times \Sp^2$,
where $\Delta_\omega=\sum_{k=1}^3 \Omega_k^2$ is the Laplace-Beltrami operator
on $\Sp^2$. We put
\begin{equation}
\label{DefVj}
{v}_j(t,r,\omega):=D_{-}^{(c_j)}\bigl(r{u}_j^1(t, r\omega)\bigr).
\end{equation}
Then \eqref{OurSysM} and \eqref{Polar} lead to
\begin{equation}
\label{ReducedEq}
\pa_{+}^{(c_j)}{v}_j(t,r,\omega)=G_j(t,r,\omega), 
\end{equation}
where 
$$
G_j(t,r,\omega):=-\frac{1}{2c_j}\left\{c_j^2 r^{-1}\Delta_\omega {u}_j^1(t, r\omega)+r
F_j\bigl(\pa u(t,r\omega), \pa^2 u(t,r\omega)\bigr)\right\}.
$$

We suppose $(t, r)\in \Lambda_0$ and $\omega\in \Sp^2$ for a while. Note that we have
$$
\jb{1+2c_1^{-1}}^{-1} \jb{t+r}\le r\le \jb{r}\le \jb{t+r},\quad (t, r)\in \Lambda_0.
$$
From \eqref{Basic01a'} we obtain 
\begin{equation}
\label{Est01}
|r^{-1}\Delta_\omega {u}_j^1(t,r\omega)|\le
C\ve^2 \jb{t+r}^{-2+\delta} \jb{c_jt-r}^{-\delta}.
\end{equation}
Recall the definitions of $N_j$, $R_j^{\rm I}$, and $R_j^{\rm II}$
given in \eqref{DefNullT}, \eqref{DefNonResI}, and \eqref{DefNonResII},
respectively.
By Lemma~\ref{NullEst}, \eqref{Basic01b}, and \eqref{Basic01a'},
we get 
\begin{equation}
\label{Est02}
r|N_j(\pa u, \pa^2 u)|\le C\ve^2\jb{t+r}^{-2+\delta}\jb{c_jt-r}^{-1-\delta}.
\end{equation}
Noting that we have
$$
\jb{c_kt-r}\jb{c_lt-r}\ge C \jb{t+r}\min\{\jb{c_k t-r}, \jb{c_l t-r}\}
$$
for $c_k\ne c_l$, we obtain from \eqref{Basic01b} that
\begin{align}
r|R_j^{\rm I}(\pa u, \pa^2 u)|\le & C\ve^2 \jb{t+r}^{-1}
\sum_{k\ne l} \jb{c_kt-r}^{-1}\jb{c_lt-r}^{-1}
\nonumber\\
\le & C\ve^{2}\jb{t+r}^{-2} \sum_{k=1}^N \jb{c_kt-r}^{-1}. 
\label{Est03}
\end{align}
By \eqref{Basic01b} we also have
\begin{equation}
\label{Est04}
r|R_j^{\rm II}(\pa u, \pa^2 u)|\le C\ve^2 
\sum_{k\ne j} \jb{t+r}^{-1} \jb{c_kt-r}^{-2}.
\end{equation}
Now \eqref{Est01}, \eqref{Est02}, \eqref{Est03}, and \eqref{Est04} lead to
\begin{align}
|G_j(t,r,\omega)|\le & C\ve^2 \jb{t+r}^{-2+\delta} \jb{c_jt-r}^{-\delta}
\nonumber\\
& 
{}+C\ve^2\sum_{k\ne j} \jb{t+r}^{-1} \jb{c_k t-r}^{-2}.
\label{Est05}
\end{align}

We define 
\begin{equation}
\label{Construct01}
{V}_j(\sigma,\omega)={v}_j\bigl(
 t_{0,j}(\sigma),r_{0,j}(\sigma),\omega
\bigr)+\int_{t_{0,j}(\sigma)}^\infty
G_j(s, c_js+\sigma, \omega) ds,
\end{equation}
where $t_{0,j}(\sigma)$ is defined by \eqref{Deft0i}, and $r_{0,j}(\sigma)=c_jt_{0,j}(\sigma)+\sigma$.
Since we have \eqref{ReducedEq} and \eqref{Est05}, Lemma~\ref{KeyLemma} implies
\begin{align}
|{v}_j(t, r, \omega)-{V}_j(r-c_jt,\omega)|\le &
C\ve^2 \left(\jb{t+r}^{-1+\delta}\jb{c_jt-r}^{-\delta} 
+\jb{t+r}^{-1}\right)
\nonumber\\
\le & C\ve^2\jb{t+r}^{-1+\delta}\jb{c_jt-r}^{-\delta}
\label{Est06}
\end{align}
for $(t,r)\in \Lambda_0$ and $\omega\in \Sp^2$.
By \eqref{DefVj}, \eqref{Basic01b}, and \eqref{Basic01a'}, we get
\begin{align}
|{v}_j(t, r, \omega)|
\le & C\left(r|\pa u_j^1(t,r\omega)|+|u_j^1(t,r\omega)|\right)
\nonumber\\
\le & C\ve^2
\left(\jb{c_jt-r}^{-1}+\jb{t+r}^{-1+\delta}\jb{c_jt-r}^{-\delta}\right)
\nonumber\\
\le & C\ve^2 \jb{c_j t-r}^{-1}.
\end{align}
Hence, putting $r=c_jt+\sigma$ in \eqref{Est06}, we get
\begin{align}
|V_j(\sigma, \omega)|\le & |v_j(t, c_jt+\sigma, \omega)|+C\ve^2\jb{(c_j+1)t+\sigma}^{-1+\delta}
\jb{\sigma}^{-\delta}\nonumber\\
\le & C\ve^2\bigl(\jb{\sigma}^{-1}+\jb{(c_j+1)t+\sigma}^{-1+\delta}
\jb{\sigma}^{-\delta}\bigr)
\label{BlackCat}
\end{align}
for all $(\sigma, \omega)\in \R\times \Sp^{n-1}$ and $t\ge t_{0,j}(\sigma)$.
Taking the limit in \eqref{BlackCat} as $t\to\infty$, 
we obtain
\begin{equation}
\label{Est07}
|{V}_j(\sigma, \omega)|\le C\ve^2 \jb{\sigma}^{-1},\quad (\sigma, \omega)\in \R\times \Sp^2.
\end{equation}
Now we define ${P}_j^1(\sigma,\omega)=\ve^{-1}{V}_j(\sigma, \omega)$.
Then, recalling the definition of ${v}_j$,
we obtain from \eqref{frame04}, \eqref{Basic01b}, and \eqref{Basic01a'}
that
$$
\left|r\pa {u}_j^1(t,x)-\vec{\omega}_{c_j}(x){v}_j(t, r, \omega)\right|
\le C\ve^2 \jb{t+r}^{-1+\delta}\jb{c_jt-r}^{-\delta},\quad (t,x)\in \Lambda.
$$
Therefore \eqref{Est06} leads to
\begin{equation}
\label{NonlinApprox01}
\left|r\pa {u}_j^1(t,x)-\ve \vec{\omega}_{c_j}(x){P}_j^1(r-c_jt,\omega)\right|
\le C\ve^2 \jb{t+r}^{-1+\delta}\jb{c_jt-r}^{-\delta}
\end{equation}
for $(t,x)\in \Lambda$.
Finally we define
\begin{equation}
P_j(\sigma, \omega)=\pa_\sigma U_j^0(\sigma, \omega)
{}+{P}_j^1(\sigma,\omega).
\label{FiveTwentyFive}
\end{equation}
Then \eqref{FreeApprox} and \eqref{NonlinApprox01} yield \eqref{Main11}
for $(t,x)\in \Lambda$.
Noting that \eqref{Est07} and the definition of $P_j^1$ lead to
\begin{equation}
\label{PseudoTwoFive}
|P_j^1(\sigma, \omega)|\le C\ve \jb{\sigma}^{-1},\quad (\sigma, \omega)\in \R\times \Sp^{2},
\end{equation}
and recalling \eqref{SupportUi0}, we immediately obtain \eqref{Main12} by \eqref{FiveTwentyFive}.

Now what is left to prove is \eqref{Main11} for $(t,x)\in \Lambda^{\rm c}$.
Since we have either $t\le 2c_1^{-1}$ or $r\le c_1t/2(\le c_jt/2)$ in $\Lambda^{\rm c}$, we get $\jb{c_jt-r}^{-1}\le C\jb{t+r}^{-1}$. 
Therefore, by \eqref{Basic01b} we get
\begin{equation}
\label{KL01}
|r\pa u_j(t,x)|\le C\ve\jb{c_jt-r}^{-1}\le 
C\ve\jb{t+r}^{-1+\delta}\jb{c_jt-r}^{-\delta}
\end{equation}
for $(t,x)\in \Lambda^{\rm c}$.
It follows from \eqref{SupportUi0}, \eqref{PseudoTwoFive}, and \eqref{FiveTwentyFive} that
$$
|P_j(\sigma, \omega)|\le C\jb{\sigma}^{-1},\quad (\sigma, \omega)\in \R\times \Sp^2,
$$
which immediately yields
\begin{equation}
\label{KL02}
|P_j(r-c_jt,\omega)|\le C\jb{c_jt-r}^{-1}\le C\jb{t+r}^{-1+\delta}\jb{c_jt-r}^{-\delta}
\end{equation}
for $(t,x)\in \Lambda^{\rm c}$.
From \eqref{KL01} and \eqref{KL02} we obtain 
$$
|r\pa u_j(t, x)-\ve \vec{\omega}_{c_j}(x)P_j(r-c_jt,\omega)|\le C\ve \jb{t+r}^{-1+\delta}\jb{c_jt-r}^{-\delta},
\quad (t, x)\in \Lambda^{\rm c},
$$
which is
\eqref{Main11} for 
$(t,x)\in \Lambda^{\rm c}$. This completes the proof.
\end{proof}
\begin{proof}[Proof of Theorem~$\ref{Main01}$ $(2)$] 
Suppose that $R^{\rm II}$ has the null structure,
and fix $\rho\in (1/2, 1)$.
Then following the proof
of the global existence theorem in \cite{Kat05}
(see Proposition~5.2 and its proof in \cite{Kat05} specifically), we have 
\begin{align}
|{u}_k(t,x)|_2+\ve^{-1}|{u}_k^1(t,x)|_2\le & C\ve \jb{t+r}^{-1}\jb{c_kt-r}^{-1},
\label{Basic02a}
\\
|\pa {u}_k(t,x)|_2+\ve^{-1}|\pa {u}_k^1(t,x)|_2\le & C\ve \jb{r}^{-1}\jb{c_kt-r}^{-1-\rho},
\label{Basic02b00}
\end{align}
which are better than \eqref{Basic01a} and \eqref{Basic01b}.
Using decay estimates in \cite{KatYok06} (see Lemmas~3.2 and 6.1 in \cite{KatYok06}), we can further improve \eqref{Basic02b00} in the decay rate and we get
\begin{equation}
|\pa {u}_k(t,x)|_1+\ve^{-1}|\pa {u}_k^1(t,x)|_1\le C\ve \jb{r}^{-1}\jb{c_kt-r}^{-2}.
\label{Basic02b}
\end{equation}
Indeed, similarly to \eqref{Est02} and \eqref{Est03}, we obtain from \eqref{Basic02a} and \eqref{Basic02b00}
that 
\begin{align*}
r |N_j|_1 \le & C\ve^2\left(\jb{t+r}^{-1}\jb{r}^{-1}\jb{c_jt-r}^{-2-\rho}+\jb{r}^{-2}\jb{c_jt-r}^{-1-2\rho}\right)\\
\le & C\ve^2 \jb{t+r}^{-2}w_-(t,r)^{-1-2\rho},\\
r|R_j^{\rm I}|_1\le & C\ve^2 \jb{r}^{-1}\sum_{k\ne l} \jb{c_kt-r}^{-1-\rho}\jb{c_lt-r}^{-1-\rho}\\
\le & C\ve^2 \jb{t+r}^{-2} w_-(t,r)^{-1-2\rho}
\end{align*}
at $(t,x)\in (0,\infty)\times \R^3$,
where $w_-(t,r)=\min_{0\le k\le N} \jb{c_kt-r}$ with $c_0=0$. Here we have used the estimates
$\jb{r}\jb{c_jt-r}\ge C \jb{t+r}w_-(t,r)$ and
$$
\jb{r}\jb{c_kt-r}\jb{c_lt-r}\ge C\jb{t+r}^2w_-(t,r),\quad k\ne l.
$$
Since $R^{\rm II}$ has the null structure, we see that $R^{\rm II}$ enjoys the same estimate
as $N_j$.
Now we have proved that 
\begin{equation}
r|F(\pa u, \pa^2 u)|_1\le C\ve^2 \jb{t+r}^{-2}w_-(t,r)^{-1-2\rho},
\label{Kiririn}
\end{equation}
and we obtain \eqref{Basic02b} from the following estimate,
which comes from Lemmas~3.2 and 6.1 in \cite{KatYok06}
(especially see the estimates (3.7), (6.1) and (6.2) in \cite{KatYok06}):
For $c, \kappa, \mu>0$, it holds that
\begin{align*}
& \jb{r}\jb{ct-r}^{1+\kappa}|\pa \phi(t,x)|\\
& \qquad \le C\sup_{y\in \R^3} \jb{y}^{2+\kappa}
 |\pa \phi(\tau,y)|_1\bigr|_{\tau=0}\\
& \qquad\quad {}+C\sup_{(\tau,y)\in[0,t]\times \R^3} |y|\jb{\tau+|y|}^{1+\kappa}w_-(\tau, |y|)^{1+\mu}
 |(\dal_c\phi)(\tau,y)|_1.
\end{align*}
 
Let $(t,r)\in \Lambda_0$ and $\omega\in \Sp^2$ for a while. Recall the proof of (1). 
We use \eqref{Basic02a} and \eqref{Basic02b}, instead of \eqref{Basic01a} (or \eqref{Basic01a'})
and \eqref{Basic01b}.
Then \eqref{Est01}, \eqref{Est02}, and \eqref{Est03} are replaced by
\begin{align*}
|r^{-1}\Delta_\omega{u}_j^1(t, r\omega)|\le  & C\ve^2\jb{t+r}^{-2}\jb{c_jt-r}^{-1},\\
r|N_j(\pa u, \pa^2 u)|\le & C\ve^2 \jb{t+r}^{-2}\jb{c_jt-r}^{-3},\\
r|R_j^{\rm I}(\pa u, \pa^2 u)|\le & C\ve^2 \jb{t+r}^{-3}\sum_{k=1}^N\jb{c_kt-r}^{-2},
\end{align*}
respectively. Since $R^{\rm II}$ has the null structure, we can use Lemma~\ref{NullEst} to get
$$
r|R_j^{\rm II}(\pa u, \pa^2 u)|\le C\ve^2\sum_{k\ne j}\jb{t+r}^{-2}\jb{c_kt-r}^{-3},
$$
instead of \eqref{Est04}.
These estimates lead to
\begin{equation}
|G_j(t,r,\omega)|\le C\ve^2 \jb{t+r}^{-2}\left( \jb{c_jt-r}^{-1}
{}+\sum_{k\ne j}\jb{c_kt-r}^{-3}\right)
\end{equation}
for $(t,r)\in \Lambda_0$ and $\omega\in \Sp^2$.
Now going similar lines to \eqref{Construct01} through \eqref{PseudoTwoFive},
we can construct ${P}_j^1$ satisfying
\begin{equation}
\label{NonlinApprox02}
\left|r\pa {u}_j^1(t, x)-\ve \vec{\omega}_{c_j}(x) {P}_j^1(r-c_jt, \omega)\right|\le C\ve^2 \jb{t+r}^{-1}\jb{c_jt-r}^{-1}
\end{equation}
for $(t,x)\in \Lambda$, and
\begin{equation}
\label{Est31}
|{P}_j^1(\sigma, \omega)|\le C\ve \jb{\sigma}^{-2},\quad (\sigma, \omega)\in \R\times \Sp^2.
\end{equation}
We define
\begin{equation}
\label{DefUi}
{U}_j^1(\sigma, \omega):=-\int_\sigma^\infty {P}_j^1(\lambda,\omega) d\lambda
\end{equation}
so that we have $\pa_\sigma {U}_j^1(\sigma, \omega)={P}_j^1(\sigma, \omega)$.
Note that the right-hand side of \eqref{DefUi} is finite because of \eqref{Est31},
and that $U_j^1(\sigma, \omega)\to 0$ as $\sigma\to\infty$.
It follows from \eqref{NonlinApprox02} and \eqref{Basic02a} that
\begin{equation}
\label{NonlinApprox03}
\left|\pa_r\bigl(r {u}_j^1(t, r\omega)\bigr)-\ve (\pa_\sigma {U}_j^1)(r-c_jt, \omega)\right|
\le C\ve^2 \jb{t+r}^{-1}\jb{c_jt-r}^{-1}
\end{equation}
for $(t,x)\in \Lambda$.
Since \eqref{Basic02a} implies $ru_j^1(t,r\omega)\to 0$ as $r\to \infty$,
from \eqref{NonlinApprox03} we get
\begin{align}
|r{u}_j^1(t,r\omega)-\ve {U}_j^1(r-c_jt, \omega)|\le & \left|-\int_r^\infty 
\pa_\lambda\bigl(\lambda {u}_j^1(t,\lambda\omega)-\ve {U}_j^1
(\lambda-c_jt, \omega)\bigr)d\lambda \right|
\nonumber\\
\le & C\ve^2\int_r^\infty (1+t+\lambda)^{-1}(1+|c_jt-\lambda|)^{-1}d\lambda. \label{Lost}
\end{align}
If $r\ge 3c_jt/2$, then we have
$$
\int_r^\infty (1+t+\lambda)^{-1}(1+|c_jt-\lambda|)^{-1}d\lambda\le C\int_r^\infty(1+t+\lambda)^{-2}d\lambda
\le C(1+t+r)^{-1}.
$$
For $r$ satisfying $1\le c_1t/2\le r\le 3c_jt/2$, we get
\begin{align*}
\int_r^\infty (1+t+\lambda)^{-1}(1+|c_jt-\lambda|)^{-1}d\lambda \le &
C(1+t)^{-1}\int_{c_1t/2}^{3c_jt/2}(1+|c_jt-\lambda|)^{-1}d\lambda\\
& +C\int_{3c_jt/2}^\infty(1+t+\lambda)^{-2}d\lambda \nonumber\\
\le & C(1+t+r)^{-1}\log(2+t).
\nonumber
\end{align*}
Now it follows from \eqref{Lost} that
\begin{equation}
\label{NonlinApprox04}
\left|r {u}_j^1(t, x)-\ve {U}_j^1(r-c_jt, \omega)\right| \le C\ve^2 \jb{t+r}^{-1}\log(2+t),\quad
(t,x)\in \Lambda.
\end{equation}
For $(\sigma, \omega)\in \R\times \Sp^2$, by \eqref{NonlinApprox04} and \eqref{Basic02a} we get
$$
|{U}_j^1(\sigma, \omega)|\le C\ve \jb{(c_j+1)t+\sigma}^{-1}\log(2+t)+C\ve\jb{\sigma}^{-1}
$$
for sufficiently large $t$. Now taking the limit as $t\to\infty$, we obtain
\begin{equation}
\label{Est32}
|{U}_j^1(\sigma, \omega)|\le C\ve \jb{\sigma}^{-1}.
\end{equation}
Finally, putting
$$
U_j(\sigma, \omega)=U_j^0(\sigma,\omega)+{U}_j^1(\sigma,\omega),
$$
we obtain \eqref{Main31} and \eqref{Main32} for $(t,x)\in \Lambda$ from \eqref{FreeApprox}, \eqref{NonlinApprox02},
and \eqref{NonlinApprox04}.  \eqref{Main33} is an immediate consequence of \eqref{Est31} and
\eqref{Est32}. 

As in the proof of (1), using \eqref{Basic02a}, \eqref{Basic02b}, \eqref{Est31}, and \eqref{Est32},
as well as \eqref{SupportUi0}, we can show that \eqref{Main31} and \eqref{Main32} hold also for $(t,x)\in \Lambda^{\rm c}$.
This completes the proof.
\end{proof}

Now we are in a position to prove Corollary~\ref{AsympFreeEnergy}.
\begin{proof}[Proof of Corollary~$\ref{AsympFreeEnergy}$]
Let the assumptions of Corollary~$\ref{AsympFreeEnergy}$ be fulfilled.
Let $P_j$ for $1\le j\le N$ be from Theorem~\ref{Main01} (1).
Recalling \eqref{SupportUi0}, we see from \eqref{Main12} that 
$|P_j(\sigma, \omega)|\le C \jb{\sigma}^{-1}$, which implies
$P_j\in L^2(\R\times \Sp^2)$.
If we put $P_j^\sharp(t,x)=|x|^{-1}P_j(|x|-c_jt, |x|^{-1}x)$, it follows from
\eqref{Main11} that
\begin{align*}
& \bigl\|\pa u_j(t,\cdot)-\ve \vec{\omega}_{c_j}(\cdot)P_j^\sharp(t,\cdot)
\bigr\|_{L^2(\R^3)}^2\\
& \quad = \int_0^\infty \int_{\Sp^2} |r\pa u_j(t,r\omega)-\ve\vec{\omega}_{c_j}(r\omega)P_j(r-c_jt,\omega)|^2 dS_\omega dr\\
& \quad \le C\ve^2 \int_0^\infty \jb{t+r}^{-2+2\delta}\jb{c_jt-r}^{-2\delta} dr
\le C\ve^2(1+t)^{-1} \to 0,\quad t\to\infty.
\end{align*}
Now Theorem~\ref{AFSE}
(with $n=3$, $c=c_j$, $v=u_j$, and $V=\ve P_j$) 
implies that each $u_j$ is asymptotically free in the energy sense.
This completes the proof.
\end{proof}
\section{Proof of Theorem~\ref{NotFree}}\label{P02}
Suppose that $\ve$ is sufficiently small, and
let $u$ be the global solution to \eqref{OurExample} 
with initial data $u=\ve f$ and $\pa_t u=\ve g$ at $t=0$.
Since the last half of \eqref{Main41} follows from $\eqref{Basic01a}$
for any $f, g\in C^\infty_0(\R^3)$,
we will prove the existence of $(f,g)$ for which the first half holds.
Without loss of generality, we may assume $A_1>0$ and $A_2\ge 0$.
We may also assume $c_1=1<c=c_2$.

We suppose that $f\equiv 0$, and that $g=(g_1,g_2)$ is radially symmetric, i.e.,
$g_j(x)={g}_j^*(|x|)$ with some function ${g}_j^*$ for $j=1,2$. 
Then $u=(u_1,u_2)$ is also radially symmetric in $x$-variable,
i.e., $u_j(t,x)={u}_j^*(t, |x|)$ with some ${u}_j^*(t,r)$ for $j=1,2$.
For $r\in \R$, we put
$$
h_j(r):=\frac{r}{2}{g}_j^*(|r|),\ v_j(t,r):=r {u}_j^*(t,|r|),\quad j=1,2.
$$
From \eqref{OurExample} we obtain
\begin{equation}
\label{Exp01}
v_j(t,r)=\frac{\ve}{c_j}\int_{r-c_jt}^{r+c_jt} h_j(\lambda)d\lambda
{}+\frac{1}{c_j}\int_0^t\left(\int_{r-c_j(t-\tau)}^{r+c_j(t-\tau)}G_j(\tau, \lambda)d\lambda\right)d\tau
\end{equation}
for $j=1,2$, where
\begin{align*}
G_1(t,r)=& \frac{A_1r}{2} \bigl(\pa_t {u}_2^*(t, |r|)\bigr)^2
=\frac{A_1}{2r}\bigl(\pa_t v_2(t,r)\bigr)^2,\\
G_2(t,r)=& \frac{A_2r}{2} \bigl(\pa_t {u}_1^*(t, |r|)\bigr)^2
=\frac{A_2}{2r}\bigl(\pa_t v_1(t,r)\bigr)^2.
\end{align*}

Assume that $g_2^*$ is a nonnegative function.
Since we have
\begin{align*}
\pa_tv_2(t,r)=& \ve\left\{h_2(r+c t)+h_2(r-c t)\right\} \\
& {}+\int_0^t\left\{G_2\bigl(\tau,r+c(t-\tau)\bigr)+G_2\bigl(\tau,r-c(t-\tau)\bigr)\right\}d\tau,
\end{align*}
we get
\begin{equation}
\label{Exp02}
\pa_t v_2(t, r)\ge \ve h_2(r-ct)(\ge 0),
\end{equation}
provided that $r-ct\ge 0$.
We fix $M>0$, and 
assume that 
$$
 C_0:=\int_{0}^M |h_2(\lambda)|^2d\lambda>0.
$$

We suppose that $0\le \sigma=r-t\le M$ in what follows.
We put
$$
\tau_0=\frac{r-t}{c-1}=\frac{\sigma}{c-1},\quad \tau_1=\frac{r+t-M}{c+1}=\frac{2t+\sigma-M}{c+1},
$$
so that we have $c\tau_0=r-t+\tau_0$ and $c\tau_1+M=r+t-\tau_1$.
Let $2t\ge (c-1)^{-1}(c+1)M$ hold, so that we have $\tau_1\ge \tau_0$.
Then it is easy to see that ${\mathcal E}\subset {\mathcal D}$, where
\begin{align*}
{\mathcal D}=&\{(\tau,\lambda);\, 0\le \tau\le t,\, r-t+\tau\le \lambda\le r+t-\tau\},\\
{\mathcal E}=&\{(\tau,\lambda);\, \tau_0\le \tau\le \tau_1,\, c\tau\le \lambda\le c\tau+M\}.
\end{align*}
Therefore, noting that we have $G_1(t,r)\ge 0$ for $r>0$, we get
\begin{align*}
\iint_{{\mathcal D}} G_1(\tau, \lambda) d\tau d\lambda
\ge & \iint_{{\mathcal E}} G_1(\tau, \lambda) d\tau d\lambda.
\end{align*}
From the definition of $G_1$ and \eqref{Exp02}, we obtain
\begin{align}
\iint_{{\mathcal E}} G_1(\tau, \lambda) d\tau d\lambda
\ge & \frac{A_1\ve^2}{2}\int_{\tau_0}^{\tau_1}
\left(\int_{c\tau}^{c\tau+M} \frac{|h_2(\lambda-c\tau)|^2}{\lambda}d\lambda\right)d\tau\nonumber\\
\ge & \frac{A_1\ve^2}{2}\int_{\tau_0}^{\tau_1} \frac{1}{c\tau+M}
\left(\int_0^M |h_2(\lambda)|^2d\lambda\right)d\tau \nonumber\\
=& \frac{C_0A_1\ve^2}{2c}\log \frac{c\tau_1+M}{c\tau_0+M}
\ge \frac{C_0A_1\ve^2}{2c}\log \frac{(c-1)(2ct+M)}{(c+1)(2c-1)M}.
\label{Exp05}
\end{align}

Now we choose a nonnegative function $g_1^*$ satisfying
$\int_{M}^{M+1} h_1(\lambda)d\lambda\ge 1$.
Then we get
\begin{equation}
\label{Exp06}
\int_{r-t}^{r+t} h_1(\lambda)d\lambda\ge 1
\end{equation}
for $(t,r)$ with $0\le r-t\le M$ and $t \ge M+1$

From \eqref{Exp01} with $j=1$, \eqref{Exp05}, and \eqref{Exp06}, we obtain
$$
v_1(t,r)\ge C\bigl(\ve+\ve^2\log(2+t)\bigr)
$$
for $0\le r-t\le M$ and $t\gg 1$. This completes the proof. \qed
\section*{Acknowledgments}
This work is partially supported by Grant-in-Aid for Scientific Research (C)
(No.~20540211 and No.~23540241), JSPS.

\end{document}